\documentclass[11pt]{amsart}

\usepackage{color}
\usepackage{amssymb}
\usepackage{graphics}
\usepackage{graphicx}
\usepackage{amscd}
\usepackage{amsmath, amsthm, epsfig,  psfrag}
\usepackage{pinlabel}

\newtheorem{theorem}{Theorem}[section]
\newtheorem{lemma}{Lemma}[section]
\newtheorem{prop}{Proposition}[section]
\newtheorem{defn}{Definition}[section]
\newtheorem{cor}{Corollary}[section]

\theoremstyle{definition}
\newtheorem{remark}{Remark}[section]

\newcommand{\nbd}{\nu}
\newcommand{\relative}{relative\ }
\newcommand{\HF}{\widehat{HF}}
\newcommand{\CF}{\widehat{CF}}
\newcommand{\HFK}{\widehat{HFK}}
\newcommand{\CFK}{\widehat{CFK}}
\newcommand{\on}{\operatorname}
\newcommand{\PD}{\on{PD}}

\renewcommand{\d}{\partial}
\newcommand{\relhom}{\alpha}
\newcommand{\Spinc}{\on{Spin}^c}
\newcommand{\gr}{\on{A}}

\newcommand{\aalpha}{\mbox{\boldmath$\alpha$}}
\newcommand{\bbeta}{\mbox{\boldmath$\beta$}}
\newcommand{\x}{{\bf x}}
\newcommand{\y}{{\bf y}}
\newcommand{\F}{\mathcal{F}}

\newcommand{\FF}{\mathbb{F}}
\newcommand{\D}{\mathcal{D}}
\renewcommand{\P}{\mathcal{P}}

\newcommand{\ZZ}{\mathbb{Z}}
\newcommand{\T}{\mathbb{T}}
\newcommand{\RR}{\mathbb{R}}

\newcommand{\Q}{\mathbb{Q}}
\newcommand\goth[1]{\mathfrak{#1}}
\newcommand{\s}{\goth{s}}

\newcommand{\Sym}{\mathrm{Sym}}
\newcommand{\ti}{\tilde}
\newcommand{\coren}{K_n}
\newcommand{\corezero}{K_0}

\newcommand{\Yn}{Y_n}
\newcommand{\Yzero}{Y_0}
\newcommand{\iotan}{{\iota}^n}
\newcommand{\deltan}{{\delta}^n}
\newcommand{\nbdK}{\nbd{K}}
\newcommand{\nbdKprime}{\nbd{K'}}

\newcommand{\Z}{\mathbb{Z}}
\def\endproofof{\relax\ifmmode\expandafter\endproofmath\else
  \unskip\nobreak\hfil\penalty50\hskip.75em\hbox{}\nobreak\hfil\bull
  {\parfillskip=0pt \finalhyphendemerits=0 \bigbreak}\fi}
\def\endproofofmath$${\eqno\bull$$\bigbreak}

\def\bull{\vbox{\hrule\hbox{\vrule\kern3pt\vbox{\kern6pt}\kern3pt\vrule}\hrule}}

\newcommand{\OneHalf}{\frac{1}{2}}

\newcommand{\cm}{\cdot}

\newcommand\SpinC{\mathrm{Spin}^c}

\newcommand\RelSpinC{\underline{\SpinC}}
\newcommand\relspinc{\underline{\spinc}}

\newcommand\Filt{\mathcal F}

\newcommand\ModSphere{\ModFlow\left({\mathbb S}\longrightarrow 
\Sym^{g-1}(\Sigma_{1})\times \Sym^2(\Sigma_{2})\right)}
\newcommand\ModSpheres\ModSphere

\newcommand\CFa{\widehat{CF}}

\newcommand\HFa{\widehat{HF}}

\newcommand\UnparModSp{\widehat \ModSp}
\newcommand\UnparModFlow\UnparModSp
\newcommand\Mod\ModSp

\newcommand{\spinc}{\mathfrak s}

\newcommand{\spinct}{\mathfrak t}

\newcommand\ModMaps{\mathcal M}
\newcommand\ModSp\ModMaps

\newcommand\Ta{{\mathbb T}_{\alpha}}
\newcommand\Tb{{\mathbb T}_{\beta}}

\newcommand\alphas{\mbox{\boldmath$\alpha$}}

\newcommand\betas{\mbox{\boldmath$\beta$}}

\newcommand\spincrel\relspinc

\newcommand\Dual{\mathcal D}
\newcommand\Duality\Dual

\newcommand\ons{Ozsv{\'a}th and Szab{\'o}}
\newcommand\os{{Ozsv{\'a}th-Szab{\'o}}}

\begin{document}

\author[Matthew Hedden]{Matthew Hedden} \thanks{MH is partially supported by NSF grant DMS-0906258 and an A.P. Sloan Research Fellowhip}
\address{Department of Mathematics, Michigan State University, East Lansing, MI 48823}
\email{mhedden@math.msu.edu}

\author[Olga Plamenevskaya]{Olga Plamenevskaya} 
\thanks{OP is partially supported by NSF grant DMS-0805836}
\address{Department of Mathematics, Stony Brook University, Stony Brook, NY 11790}
\email{olga@math.sunysb.edu}
\title[Dehn surgery, rational open books, and knot Floer homology]{Dehn surgery, rational open books, \\ and knot Floer homology}

\begin{abstract} By recent results of Baker--Etnyre--Van Horn-Morris, a rational open book decomposition defines a compatible contact structure.  
We show that the Heegaard Floer contact invariant of such a contact structure can be computed 
in terms of the knot Floer homology of its (rationally null-homologous) binding.
 We then use this description of contact invariants, together with a  formula for the knot Floer homology of the core of a surgery solid torus, 
to show that certain manifolds obtained by surgeries on bindings of open books carry tight contact structures. 
\end{abstract}
\maketitle

\section{Introduction} 

Dehn surgery is the process of excising a neighborhood of an embedded circle 
   (a knot) in a $3$-dimensional manifold and subsequently regluing it with a diffeomorphism of the bounding torus. 
 This construction has long played a fundamental role in the study of $3$-manifolds, and provides a complete method of construction.  If the $3$-manifold is equipped with  extra structure, one
 can hope to adapt the surgery procedure to incorporate this structure.   
This idea has been fruitfully employed in a variety of situations.  
 
 Our present interest lies in the realm of $3$-dimensional contact geometry. 
 Here, {\em Legendrian} (and  more recently, {\em contact})   surgery has been an invaluable tool for the 
study of $3$-manifolds equipped with a contact structure (i.e. a completely non-integrable two-plane field).  
For a contact surgery on a {\em Legendrian} knot, we  start with  a knot which is tangent to the contact structure,
 and perform Dehn surgery in such a way that the contact structure on the knot complement 
is extended over the surgery solid torus \cite{DGS}. 
 To guarantee that the extension is unique, a condition on the surgery slope is required. Namely, the slope 
must differ from the contact framing by $\pm \text{meridian}$; this gives two types of surgery, Legendrian surgery 
(aka $(-1)$ contact surgery) and its inverse, $(+1)$ contact surgery.

A central goal of this article is to study a different situation in 
which Dehn surgery uniquely produces a contact manifold.  
 For this we employ an important tool in $3$-dimensional contact geometry: open book decompositions.  
 An {\em open book decomposition} of a $3$-manifold $Y$ is equivalent to a choice of {\em fibered} knot
  $K\subset Y$, by which we mean a knot whose complement fibers over the circle so that the boundary of 
any fiber is a longitude.   We refer to $K$ as the {\em binding} of the open book.  
 From an open book decomposition, one can produce a contact structure which is unique, up to isotopy.  
Note that for this contact structure, the knot $K$ will be {\em transverse} to the contact planes. 
Surgeries on transverse knots were studied in \cite{Ga}, but our perspective is different from \cite{Ga}. 

Given a knot $K\subset Y$,  denote the manifold obtained by Dehn surgery with slope $p/q$ by $Y_{p/q}$. 
 There is a canonical knot induced by the surgery; namely, the core of the solid torus used in the construction. 
 We denote this knot by $K_{p/q}$.   If we perform surgery on a fibered knot $K\subset Y$ then the complement 
of the induced knot $K_{p/q}\subset Y_{p/q}$ fibers over the circle; indeed, it is homeomorphic to the complement
 of $K$.  However, $K_{p/q}$ is often not fibered in the traditional sense, as the boundaries of the fibers are not 
longitudes.  In fact $K_{p/q}$ will be homologically essential if $p\ne 1$, and so will not have a Seifert surface at all. 
 If $p\ne 0$, then $K$ will be {\em rationally null-homologous}, meaning that a multiple of its homology class is zero. 
 We  refer to a rationally null-homologous knot whose complement fibers over the circle as a {\em rationally fibered} 
knot, and the corresponding decomposition of the $3$-manifold as a {\em rational open book decomposition}.    
 Baker--Etnyre--Van Horn-Morris \cite{BEV} recently showed that a rational open book gives rise to a contact 
structure which is unique, up to isotopy.   Thus a  fibered knot $K\subset Y$ induces a unique contact
 structure $\xi$ on $Y$, and Dehn surgery on $K$ gives rise to a rationally fibered knot $K_{p/q}\subset Y_{p/q}$ 
inducing a unique contact structure $\xi_{p/q}$ on $Y_{p/q}$.  The purpose of this article is to investigate the
 relationship between these contact structures.

Our investigation will rely on Heegaard Floer homology, which provides a
 powerful invariant of contact structures.  Denoted $c(\xi)$, this invariant lives in $\HFa(-Y)$, 
the Heegaard Floer homology of the manifold $Y$ with its orientation reversed ($\FF=\Z/2\Z$ coefficients
 are used throughout, to avoid any sign ambiguities).  We study $\xi_{p/q}$ by way of its contact invariant, 
 so it will be useful to understand how to compute the contact invariant associated to a rational open book. 
 Our first theorem states that, as in the null-homologous case, the contact invariant is a function of 
the knot Floer homology of the binding.  



To understand the statement, recall that a rationally null-homologous knot $K\subset Y$ induces a $\Z$-filtration of $\CFa(-Y)$;  that is,  a sequence of subcomplexes with integer indices:
$$ 0\subset \F(\mathrm{bottom})\subset \F(\mathrm{bottom}+1)\subset... \subset \CFa(-Y).$$ (See Section \ref{sec:filtration} for more details on the filtration.) We have

\begin{theorem}  \label{cont-inv} Let $K \subset Y$ be a rationally fibered knot, and $\xi_K$ the contact structure induced by the associated rational open book decomposition.  Then $H_*(\F(\mathrm{bottom}))\cong \FF\cm\langle c\rangle$.  Moreover, if $$ \iota: \F(\mathrm{bottom})\rightarrow \CFa(-Y),$$ is the inclusion map of the lowest non-trivial subcomplex, then   $\iota_*(c)=c(\xi_K)\in \HFa(-Y)  $.  \end{theorem}

In the case that $K$ is fibered in the traditional sense (so that it induces an honest open book decomposition of $Y$) this agrees with  \ons's definition of $c(\xi)$.     We also remark that the definition of the filtration depends on a choice of relative homology class, and the class used in the theorem  comes from the  fiber.  

The proof of Theorem \ref{cont-inv} uses a cabling argument.   More precisely,  an appropriate cable of $K$ is a fibered knot in the traditional sense, and results of \cite{BEV} relate the contact structure of  the resulting open book to that of the original rational open book.  We  prove the theorem by  developing a corresponding understanding of the behavior of the knot filtration under cabling.  This is aided by techniques developed in \cite{He}.   We should point out that while the cabling argument shows that  $H_*(\F(\mathrm{bottom}))\cong \FF$, we  give an alternate proof of this fact by constructing an explicit Heegaard diagram adapted to a rational open book  where the subcomplex in question is generated by a single element, Proposition \ref{intersection-pt}.  This is a rational analogue of the Heegaard diagram for fibered knots constructed in \cite{contOS}, and may be useful for understanding the interaction between properties of the monodromy of a rational open book and those of the contact invariant.  By combining the theorem with results of \cite{NiFibered} and \cite{He4} (see also \cite{Ra}) we  arrive at the following corollary.

\begin{cor}\label{cor:lens} Suppose $K\subset L(p,q)$ is a knot in a lens space and that integral surgery on $K$ yields the $3$-sphere.  Then $K$ is rationally fibered and the associated rational open book induces a  contact structure, $\xi_K$, with $c(\xi_K)\ne0$. Regarding $K$ in $-L(p,q)$, the lens space with orientation reversed, we obtain a contact structure $\xi_{\overline{K}}$ also satisfying  $c(\xi_{\overline{K}})\ne0$ .     
\end{cor} 

\begin{remark}  \cite[Theorem 1.4]{contOS} shows that non-vanishing contact invariant implies  tightness, so the contact structures of the corollary are tight.  The corollary also applies to knots in L-spaces which admit homology sphere $L$-space surgeries. The proof of the corollary, contained in subsection 3.2, is based on the fact that the  Floer homology of  knots on which one can perform surgery to pass between L-spaces (manifolds with the simplest Heegaard Floer homology) is severely constrained. 
\end{remark}

 We find this corollary particularly intriguing, not due to the existence of a tight contact structure on $L(p,q)$ induced by $K$, but the additional tight contact structure on $-L(p,q)$.  To put this in perspective, if a null-homologous  fibered knot  $K\subset Y$ induces a tight contact structure on both $Y$ and $-Y$, then the monodromy of the associated open book is isotopic to the identity (otherwise it could not be right-veering with both orientations \cite{HKM2}). If one could show that, similarly, there are but a finite number of rationally fibered knots which induce tight contact structures on both $L(p,q)$ and $-L(p,q)$, this would lead to significant ---  if not complete --- progress on the Berge Conjecture (a conjectured classification of knots in the $3$-sphere admitting lens space surgeries).  Regardless, we hope that the geometric information provided by the contact structures induced by $K\subset L(p,q)$ can be of  aid in the understanding of lens space surgeries.



    


In another direction, we can use a surgery formula for knot Floer homology to understand the contact invariant of  rational open books induced by Dehn surgery. (Here, the $3$-manifolds involved do not have to be L-spaces.) 
Our second main theorem is a non-vanishing result for the contact invariant in this situation.


\begin{theorem} \label{p/q-surgery} Let $K\subset Y$ be a fibered knot with genus $g$ fiber, and $\xi$ the contact structure induced by the associated open book.  Let $K_{p/q}\subset Y_{p/q}$ be the rationally fibered knot arising   as the core of the solid torus used to construct $p/q$ surgery on $K$, and $\xi_{p/q}$ the contact  structure

 \vspace{-0.035in} \noindent  induced by the associated rational open book.  Suppose  $c(\xi)\! \! \in\!  \HF(-Y)$ is non-zero. Then $c(\xi_{p/q}) \in \HF(-Y_{p/q})$ is  non-zero for all $p/q \geq 2g$.    
\end{theorem}

Note that surgeries with sufficiently negative framings can be realized as Legendrian surgeries. If  $(Y,\xi)$ has non-trivial contact invariant, so will any contact structure obtained  by Legendrian surgery, regardless of fibering.  For this reason, producing tight contact structures on positive Dehn surgeries is typically more challenging, and explains our focus on the realm of positive slope.  We should point out, however, that our results have analogues for negative slopes which can be used to produce contact structures  with non-trivial  invariants, even in situations where the slope is larger than the maximal Thurston-Bennequin invariant.

Theorem \ref{p/q-surgery} allows to construct a number of interesting tight contact structures. First, notice that
 surgeries on the binding of an open book with trivial monodromy produce rational open book decompositions for  circle bundles over surfaces. 
Tight contact structures on circle bundles are completely classified (\cite{Ho2, Gi3}), but it is interesting to point out that an  existence result
follows immediately from Theorem \ref{p/q-surgery}: a circle bundle of Euler number $n\geq 2g$ over a surface of genus $g>0$ carries a tight contact structure with non-zero contact invariant.  To list some further families of contact manifolds whose tightness follows from Theorem \ref{p/q-surgery}, we turn to the supply of tight contact structures compatible with the genus one open books given in \cite{Ba1, Ba}. Indeed, 
tight contact structures supported by open books $(T, \phi)$ (where $T$ is a punctured torus)
are completely classified \cite{Ba1, HKM} in terms of their monodromy. All of these tight contact structures have non-vanishing contact invariants, 
so Theorem \ref{p/q-surgery} produces, for any  $p/q \geq  2$, tight contact structures on manifolds obtained by $p/q$-surgery on the bindings of corresponding open books. Many of these manifolds are L-spaces \cite{Ba} and thus carry no taut foliations \cite[Theorem 1.4]{genusOS}; 
the family of tight contact manifolds we obtain generalizes a result of Etg\"u \cite{Etg}. (Note that an expanded version of \cite{Etg} extends the 
results to a wider class of open books than the original arxiv version.)

Our results should also be contrasted to those of Lisca-Stipsicz \cite{LS}. They prove that for a knot $K\subset S^3$ 
whose maximal self-linking number equals  $2g(K)-1$, the surgered manifold $S_r^3(K)$ carries a tight contact structure for all $r \geq 2g(K)$.
While our theorem only applies to fibered knots, it can be used in arbitrary $3$-manifolds.  In particular, combining Theorem \ref{p/q-surgery} with \cite[Theorem 5]{He3} produces

\begin{cor}
Let  $K\subset Y$ be a fibered knot with fiber $F$, and $\xi$  a contact structure on $Y$ with $c(\xi)$ non-zero.  Assume  that  $K$ has a transverse representative in $\xi$ satisfying $$sl_F(K)=2g(F)-1.$$ Then $K_{p/q}\subset Y_{p/q}$ induces a contact structure $\xi_{p/q}$ with $c(\xi_{p/q})$ non-zero, for  $p/q\ge 2g(F)$.  
\end{cor}   Our result overlaps with  \cite{LS} for fibered knots in $S^3$ with $sl(K)=2g(K)-1$, but  \cite{LS}  guarantees only the existence of a tight contact structure
whereas our result describes a specific supporting open book.

Our proof of Theorem \ref{p/q-surgery} consists of several parts. The first  is based on a detailed examination of the knot Floer homology of the induced knot $K_n\subset Y_n$ for sufficiently large integral surgeries, $n\in \Z$.  Building on work of \cite{He2,knotOS}, we give a complete description of the knot Floer homology filtration induced by $K_n\subset Y_n$ in terms of the filtration induced by $K\subset Y$.   Coupled with the description of the contact invariant given by Theorem \ref{cont-inv}, this proves the theorem for $n\gg 0$.  We then obtain the theorem for all integers $n\ge 2g$ by using an exact sequence for knot Floer homology together with an adjunction inequality.  It is worth pointing out that the restriction $n\ge 2g$ is, in general, sharp (this can be seen from the $(2,k)$ torus knot).    Finally, the theorem is proved for rational slopes $p/q\ge 2g$ by showing that $\xi_{p/q}$ is  obtained from $\xi_n$ by Legendrian surgery.  


\bigskip

\noindent{\bf Outline:}  The paper is organized as follows. In Section 2, we discuss the Alexander grading in knot Floer homology, paying particular attention to the case of rationally null-homologous knots.  In particular, we discuss how to compute this grading with the help of so-called {\em \relative periodic domains}.

Section 3 is devoted the  proof of Theorem \ref{cont-inv}.   The proof relies on studying the relationship between the knot Floer homology of the binding of an open book and  that of its cables.  In this section we also produce an explicit Heegaard diagram for a rationally fibered knot with a unique generator for the lowest non-trivial filtered subcomplex in the knot filtration.  

In Section 4 we prove Theorem \ref{p/q-surgery}.  This section includes a detailed discussion of the relationship between the knot Floer homology of $K\subset Y$ and the  Floer homology of the induced knot $K_{p/q}\subset Y_{p/q}$.

\subsection*{Acknowledgements.} We are grateful to Jeremy Van Horn--Morris for many helpful conversations, to Andr{\'a}s Stipsicz and Paolo Lisca for their interest, and to John Etnyre for help with Lemma 
\ref{n-to-p/q}.   Much of this work was completed at the Mathematical Sciences Research Institute, during the program ``Homology Theories for Knots and Links", and at the Banff International Research Station during the workshop ``Interactions between contact symplectic topology and gauge theory in dimensions 3 and 4" in March, 2011.  We are very grateful for the wonderful environment provided by both institutions.  

\section{Rationally null-homologous knots and the Alexander grading} \label{sec:filtration}
 Let $K\subset Y$ be  knot.  We say that $K$ is {\em rationally null-homologous} if $[K]=0\in H_1(Y;\Q)$.  This implies that for some positive integer $p$, we have $p\cm [K]=0$ in  $H_1(Y;\Z)$, and that there exists a smooth, properly embedded surface $F\subset 
 Y\setminus \nbdK$ such that $[\partial F]=p\cm [K]$.  If $p$ is minimal, we call it the {\em order} of $K$, and refer to the aforementioned surface as a {\em rational Seifert surface} for $K$.   Finally,  we say that a rationally null-homologous knot is {\em rationally fibered} if $Y\setminus \nbdK$ fibers over the circle with fiber a rational Seifert surface.  In this section we discuss Alexander gradings in knot Floer homology, with an emphasis on the case of rationally null-homologous knots.   For such knots, an Alexander grading can be defined with the help of the relative homology class coming from a rational Seifert surface.
This Alexander grading can, in turn, be computed from a so-called {\em \relative periodic domain} which represents the homology class of the Seifert surface.

 Suppose that $K$ is a rationally null-homologous knot in $Y$, represented by a doubly-pointed Heegaard diagram $(\Sigma, \aalpha, \bbeta, w, z)$.
 The knot  induces a filtration of the chain complex $\CF(Y)$
by the partially-ordered set of relative $\Spinc$ structures  $\RelSpinC(Y,K)$ on the knot complement \cite[Section 2]{ratOS}. The partial ordering comes from the fact that  $\RelSpinC(Y,K)$ is an $H^2(Y\setminus \nbdK,\partial(Y\setminus \nbdK) )$--torsor, and this latter group can be endowed with a partial order (note  that there is no canonical  partial ordering on torsion cyclic summands in $H^2(Y\setminus \nbdK,\partial(Y\setminus \nbdK))$,  so we simply pick one).   The partial ordering restricts to a total ordering on the fibers of the natural filling map \cite[Section 2.2]{ratOS}:   
\begin{equation}\label{eq:filling} G_{Y,K}: \RelSpinC(Y,K)\longrightarrow \SpinC(Y),
\end{equation} 
where $G_{Y,K}^{-1}(\spinc)$ consists of relative $\SpinC$ structures which differ by a multiple of the Poincar{\'e} dual to the meridian   $\PD([\mu])$.

A relative homology class $\relhom\in H_2(Y\setminus \nbdK,\partial(Y\setminus \nbdK))$ allows us to collapse the partial order on $\RelSpinC$ to a total order.  Define  $\gr_\alpha: \RelSpinC(Y,K)\rightarrow\Z$ by
\begin{equation}
\gr_\alpha(\spincrel)= \frac12 \langle c_1(\spincrel)-\PD([\mu]), \alpha \rangle,
\end{equation}
where $c_1(\spincrel)\in H^2(Y\setminus \nbdK,\d (Y\setminus \nbdK)$ is the relative Chern class  of the orthogonal $2$-plane field to the relative $\SpinC$ structure, relative to a specific  trivialization on the boundary \cite[Page 627]{linksOS}.  This function gives $\RelSpinC$ a total order, and hence a total order on the set of generators for $\CF(Y)$, by the function $$\spinc_{z,w}(-): \Ta\cap\Tb\rightarrow \RelSpinC(Y,K).$$ 

For the purposes of knot Floer homology, the relevant  $\alpha\in H_2(Y\setminus \nbdK,\partial(Y\setminus \nbdK))$ is the class of a rational Seifert surface, $[F,\partial F]\in  H_2(Y\setminus \nbdK,\partial(Y\setminus \nbdK))$.    In this case, we refer to the function  
 \begin{equation}
\label{alex}
\gr_{[F,\partial F]}(\x)= \frac12 (\langle c_1(\s_{z,w}(\x)), [F, \d F] \rangle- [\mu] \cdot [F, \partial F]). 
\end{equation}
as the {\em Alexander grading}.  This depends on the choice of rational Seifert surface, but only through its relative homology class.  We will often drop this choice from the notation, letting $\gr(\x)$ denote the Alexander grading of a generator, defined with respect to an implicit choice of rational Seifert surface (when $b_1(Y)=0$ this choice is canonical).  The Alexander grading gives rise to a filtration $\F$ on $\CF(Y)$ in the standard way, i.e. we let $$
\F(s) = \underset{\{\x\in\Ta\cap\Tb\ |\ \gr(\x)\leq s\}}\bigoplus \FF<\x>, 
$$
denote the subgroup of $\CFa(Y)$ generated by intersection points with Alexander grading less than or equal to $s\in \Z$.  Positivity of intersections of $J$-holomorphic Whitney disks with the hypersurfaces determined by $z$ and $w$ ensures that $\Filt(s)$ is a subcomplex; that is, $\d \F(s)\subset \F(s)$ and hence $\F$ indeed defines a filtration.   The associated graded groups are the {\em knot Floer homology groups},  $$\HFK_*(Y,[F],K,i):=H_*\left(\frac{\F(i)}{\F(i-1)}\right).$$

The Alexander grading is slightly easier to study if $Y$ is a rational homology sphere \cite{Ni}. In this case, if 
$(Y, K)$ is represented by a doubly-pointed Heegaard diagram $(\Sigma, \aalpha, \bbeta, w, z)$, 
and $\x$, $\y$ are two generators of $\CF(Y)$, consider a  curve $a$ in $\T_{\alpha} \subset \Sym^g(\Sigma)$ 
connecting $\x$ to $\y$, and a  curve $b$ in $\T_{\beta} \subset \Sym^g(\Sigma)$ connecting $\y$ to $\x$.  
The union $a \cup b$ is a closed curve in $\Sym^g(\Sigma)$. Since $b_1(Y)=0$,
 a multiple  $k(a \cup b)$ bounds a Whitney disk $\phi$, and the filtration difference can be computed by means of 
 this Whitney disk. Indeed,
$$
\s_{w, z}(\x)- \s_{w, z}(\y)= \frac1k (n_z(\phi)- n_w(\phi)) \PD([\mu]),
$$
and this quantity is independent of $\phi$ \cite[Lemma 4.2]{Ni} (see also \cite[Lemma 3.11]{linksOS}. 

If $b_1(Y)>0$, some generators of $\CF(Y)$ may not be related by a Whitney disk, although the above formula still holds 
for $\x, \y$ such that their relative $\Spinc$-structures differ by a multiple of $\PD([\mu])$; this is always the case if $\spinc_{z,w}(\x)$ and $\spinc_{z,w}(\y)$ are in the same fiber of the filling map \eqref{eq:filling}.  To understand the Alexander grading in the absence of rational Whitney disks we will use ``\relative periodic domains" to evaluate the 
grading difference between two generators.  

Let $K\subset Y$ be a  knot, 
and let $(\Sigma, \aalpha, \bbeta, z, w)$ be a Heegaard diagram for $(Y, K)$. 
Connect $z$ to $w$ by an arc $l_1$ in $\Sigma$ disjoint from the $\alpha$-curves, 
and $w$ to $z$ by an arc $l_2$ in $\Sigma$ disjoint  from the $\beta$-curves.  The union $\lambda = l_1 \cup l_2$, when pushed into the respective handlebodies, is a longitude for $K$. We will always 
consider Heegaard diagrams where such a longitude is fixed for the given knot.    

\begin{defn} Let $K\subset Y$ be a rationally null-homologous knot, and let $(\Sigma, \aalpha, \bbeta, z, w)$
be a Heegaard diagram for $(Y, K)$ with a  longitude $\lambda$, as above.
Let $\D_1, \dots \D_r$ denote the closures of the components of 
$\Sigma \setminus ( \aalpha \cup \bbeta \cup \lambda)$. 
A {\em \relative periodic domain} is  a 2-chain $\P= \sum a_i D_i$, whose  boundary satisfies 
$$
\d \P = l \lambda + \sum_i n_i \alpha_i + \sum_i m_i\beta_i,
$$ for $l, n_i, m_i\in \Z$. 
\end{defn}

\begin{remark} Our definition is a generalization of the notion of periodic domain \cite[Definition 2.14]{threeOS}.  A periodic domain is a two chain, as above satisfying $l=0$ and  $n_w(\P)=0$. 
\end{remark}


A \relative periodic domain $\P$ naturally gives rise to a relative homology class $[\P]\in H_2(Y\setminus \nbdK,\d (Y\setminus \nbdK))$, in the same way that periodic domains give rise to homology classes in $H_2(Y)$. Indeed, a relative periodic domain is a $2$-chain whose boundary consists of a union of copies of $\lambda$ and complete $\alpha$- and $\beta$- curves.  Capping off the $\alpha$- and $\beta$- curves with the disks that they bound in their respective handlebodies, we arrive at a $2$-chain whose boundary lies on $\lambda$ or, up to homotopy, on $\partial (Y\setminus \nbdK)$.  In other words, we obtain a cycle in the relative chain group $C_2( Y\setminus \nbdK,\partial (Y\setminus \nbdK))$.  We denote the corresponding homology class by $[\P]$. In fact, the correspondence is reversible; that is, any relative homology class comes about by capping off a relative periodic domain.  Since we have no need for this fact we leave the details (a standard Mayer-Vietoris argument) to the reader.

The Alexander grading is defined in terms of the relative homology class of a rational Seifert surface.  Thus our primary interest lies in those relative periodic domains whose homology class agrees with some specific rational Seifert surface $F$.  To this end, observe that if $K$ has order $p$, then $\partial F$ will wrap $p$ times around $K$.  Thus for a relative periodic domain $\P$ whose homology class agrees with $F$, the longitude $\lambda$ will appear with multiplicity $p$ in $\d \P$.  


The following lemma shows that the relative Alexander grading difference between generators $\x,\y$ is determined by the multiplicities of $\P$.  


\begin{lemma} \label{periodic}
Let $K \subset Y$ be a rationally null-homologous knot and  $\P$ be a \relative periodic domain whose homology class equals that of a fixed rational Seifert surface $F$.  Let $\x, \y\in \T_{\alpha} \cap \T_{\beta}$.   Then 
$$
\gr(\x)-\gr(\y)=  n_{\x}(\P) -n_{\y}(\P),
$$
where $\gr$ is the Alexander grading  with respect to $F$, defined by Equation \eqref{alex}.
\end{lemma}

\begin{proof} 
Recalling the definition of  $\gr$, we need to evaluate the quantity
$$
\frac12(\langle c_1(\s_{w,z}(x)), [F, \d F] \rangle - \langle c_1(\s_{w,z}(y)), [F, \d F] \rangle)
= \langle \s_{w,z} (\x) - \s_{w,z} (\y), [F, \d F] \rangle.
$$

By \cite[Lemma 3.11]{linksOS}
$$
\s_{w,z} (\x) - \s_{w,z} (\y) = \PD( \epsilon(\x, \y))= \PD([\gamma_{\x}-\gamma_{\y}]),
$$
where $\gamma_{\x}$ is the union of gradient trajectories connecting index 1 and index 2 critical points of the Morse function
which pass through the coordinates $x_i$ of $\x=(x_1, \dots x_g)$, and $\gamma_{\y}$ is a similar union of gradient trajectories passing through the
coordinates of $\y$.
Therefore, it suffices to calculate the intersection number of the closed curve $\gamma_{\x}-\gamma_{\y}$ with the surface $F$.
To this end, recall that the homology class of $[F, \d F]$ is 
 constructed from the periodic domain $\P$ by capping off any $\alpha$- and $\beta$- curves appearing in $\partial \P$ (with multiplicity)  with the compressing disks bounded by the  curve in the corresponding handlebody.

If $x_i\in \x$ (resp. $y_i\in \y$) lies in the interior of $\P \subset \Sigma$,  
then the intersection of $F$ with $\gamma_{\x}$ (resp. $\gamma_{\y}$) equals  the multiplicity $\bar{n}_{\x_i}(\P)$ (resp. $\bar{n}_{\y_i}(\P)$). 
If $x_i \in \partial \P$ then it doesn't contribute to the intersection number, as the surface can be perturbed so that the  
compressing disk for the corresponding $\alpha$- or $\beta$-curve is replaced by a normal translate which is disjoint from $\gamma_{\x}$. 
It remains to observe that contributions from such boundary points cancel in the expression  $n_{\x}(\P) -n_{\y}(\P)$, since every $\alpha$-curve 
and every $\beta$-curve contains exactly one coordinate of $\x$ and exactly one coordinate of $\y$.
\end{proof}

\section{The contact invariant for rational open books}

Let $K\subset Y$ be a rationally fibered knot.  Such a knot induces a rational open book decomposition and, subsequently, a contact structure $\xi_K$ \cite{BEV}.  The purpose of this section is understand the \os \ contact invariant of $\xi_K$ in terms of the knot Floer homology of $K$.  More precisely,  the ``bottom" filtered subcomplex in the filtration of $\CFa(-Y)$ induced by $K$ has  homology $\FF$ (this can be seen in many ways, and follows from both Propositions \ref{cable} and \ref{intersection-pt} below).  The main result of this section, Theorem \ref{hfkbottom}, shows that
$$  c(\xi_K)=  \mathrm{Im} (\iota_*: H_*(\F(\mathrm{bottom}))\rightarrow \HFa_*(-Y) ).$$
That is, the contact invariant of  $\xi_K$ is the image of the generator of the  homology of the bottom filtered subcomplex in the Floer homology of $-Y$, under the natural inclusion-induced map. When $K$ is fibered in the traditional sense this is simply \ons's definition \cite[Definition 1.2]{contOS}.

We prove Theorem \ref{hfkbottom} by considering an honest open book which results from an appropriate cabling of $K$.  Let $K_{P,Pn+1}$ denote the $(P,Pn+1)$ cable of $K$.  It is clear  that $[K_{P,Pn+1}]=P\cm[K]\in H_1(Y;\Z)$.  Thus for $P$ equal to the order of $K$, the cables will be null-homologous.  Moreover, such cables are fibered in the traditional sense, provided that $K$ is rationally fibered.  When $P,n>0$, it follows from \cite[Theorem 1.8]{BEV} that $\xi_{K}$ is isotopic to $\xi_{K_{P,Pn+1}}$. 
 The theorem will follow by understanding the relation between the knot Floer complex  of a given knot and its cable.  This is accomplished by Proposition \ref{cable}, which generalizes the cabling result of \cite{He}.   While our results show that $H_*(\F(\mathrm{bottom}))\cong\FF$, we  conclude the section by constructing an explicit Heegaard diagram adapted to a rational open book decomposition for which this  group is represented by a complex with a single generator.  With the plan in place, we begin.

 \subsection{The contact invariant and cabling}

 
 In this subsection we prove Theorem \ref{hfkbottom}. The key tool is Proposition \ref{cable}, which establishes a relationship between the Floer complex of a rationally null-homologous knot and that of its sufficiently positive cables.
 
The result states that the knot Floer homology groups of a knot and its sufficiently positive cables are equal in the ``topmost"  Alexander gradings .  To make this precise, recall that the Alexander grading depends on a choice; namely, the relative homology class of a rational Seifert surface (Equation \eqref{alex}). To specify how this choice is made, fix a rational Seifert surface $F$  for
 the knot $K$.   We construct a rational Seifert surface $F'$ for the cable $K'= K_{P, Pn+1}$  as follows. If $K$ has order $p$ in $H_1(Y)$, 
 then $F$ intersects   $\d \nbdK$  in a curve $s$ that wraps $p$ times around the longitude.
 The cable $K'$ has order $p'=p/\mathrm{gcd}(P,p)$.  Thus a rational Seifert surface $F'$ for $K'$ must meet $\nbdKprime$ in a curve $s'$ that is null-homologous in 
 $Y\setminus \nbdKprime$ and wraps $p'$ times around the longitude.  We can assume that the  neighborhood of the cable is contained inside that of the knot, $\nbdKprime\subset \nbdK$.  To construct $F'$, we take $R=P/\mathrm{gcd}(P,p)$ parallel copies of $F$, and glue them to an oriented properly embedded surface in $\nbdK\setminus \nbdKprime$ whose boundary consists of $s'$ and $R$ parallel copies of $s$.


\begin{prop} \label{cable} 
Let $K \subset Y$ be a rationally null-homologous knot, and $K'=K_{P, Pn+1}$ its $(P, Pn+1)$-cable. Fix a rational Seifert surface $F$ 
for $K$, and consider the corresponding rational Seifert surface $F'$ for $K'$, described above. 
Then for all $n\gg0$ sufficiently large, we have
$$
\HFK_*(Y, [F], K, \mathrm{top})= \HFK_*(Y, [F'], K', \mathrm{top}).
$$
\end{prop}  

\begin{remark} Disregarding gradings, an isomorphism between the groups above can be shown for all $n$ using sutured manifold decomposition \cite[Corollary 5.9]{NiFibered} (see also \cite{JuhaszSurface}).  However, the strategy of our proof will be essential in our understanding of how the contact invariants of a rational open book and its cables are related.
\end{remark}

\begin{proof} For the case where $K$ is a knot in $S^3$, this statement was established in \cite{He} (see also \cite{Ni} for a generalization
to the case where $Y$ is a rational homology sphere). However, the proof from \cite{He} and \cite{Ni} use Whitney disks to compare the Alexander 
gradings of different generators of $\CF(Y)$; when $b_1(Y)>0$, this proof no longer works since there may be no Whitney disks.
Instead, we will use Lemma \ref{periodic} to compare the gradings in the Heegaard diagrams of \cite{He, Ni}. 

Let $F$ be a rational Seifert surface for $K$. We can find a 
 doubly-pointed Heegaard diagram $(\Sigma, \aalpha, \bbeta, w, z)$ for $(Y, K)$, together with  
a longitude  $\lambda_0$ on $\Sigma$ and a \relative periodic domain $\P$ representing $F$.  
It will be convenient to enumerate the $\alpha$- and $\beta$-curves by the index set $\{0, 1, \dots, g-1\}$, 
and to suppress the indices of $\alpha_0=\alpha$ and $\beta_0=\beta$.
We assume that $\lambda_0$ connects points $z$ and $w$, intersects $\beta_0$ 
at a single point, and is disjoint from all other $\beta$-curves, so that $\beta$ represents a meridian for $K$.
We also require that $\beta$ intersects $\alpha$ at a single point and is disjoint from the other $\alpha$-curves.
The \relative periodic domain $\P$ gives rise to the homological relation, 
\begin{equation} \label{=0}
p \lambda_0 + q \beta + \sum_{i=0}^{g-1} r_i \alpha_i + \sum_{i=1}^{g-1} q_i \beta_i=0.
\end{equation}
The multiplicities of $\P$ in the components of $\Sigma \setminus ( \aalpha  \cup \bbeta \cup \lambda_0)$ can be determined 
as follows:  pick a component $\D_0\in \Sigma \setminus ( \aalpha  \cup \bbeta \cup \lambda_0)$ and assign the multiplicity of $\P$ in $\D_0$ to be zero.  The multiplicity of $\P$ in any other component $\D_i$ is the algebraic intersection number $\# \gamma\cap \partial P$ of an oriented arc from $\D_i$ to $\D_0$ with the sum of curves in \eqref{=0}.  It is customary to fix the multiplicity 
of the component containing $w$ to be zero; however, in the argument below we find it convenient to fix the multiplicity of another component.


\begin{figure}[htb] 
	\labellist
	\small\hair 2pt
	\pinlabel $\beta$ at 93 112
	\pinlabel $\alpha$ at 48 138
	\pinlabel $\lambda_0$ at 135 66
	\pinlabel $t$ at 78 42
	\pinlabel $z'$ at 68 55
	\pinlabel $w$ at 63 87 
	\pinlabel $\lambda$ at 69 96
	\pinlabel $\Lambda$ at 111 41
	\pinlabel $x_0$ at 49 77
	\pinlabel $\alpha_i$ at 181 89
	\pinlabel $\alpha_j$ at 258 95
	\endlabellist
\centering	
\includegraphics[scale=1.5]{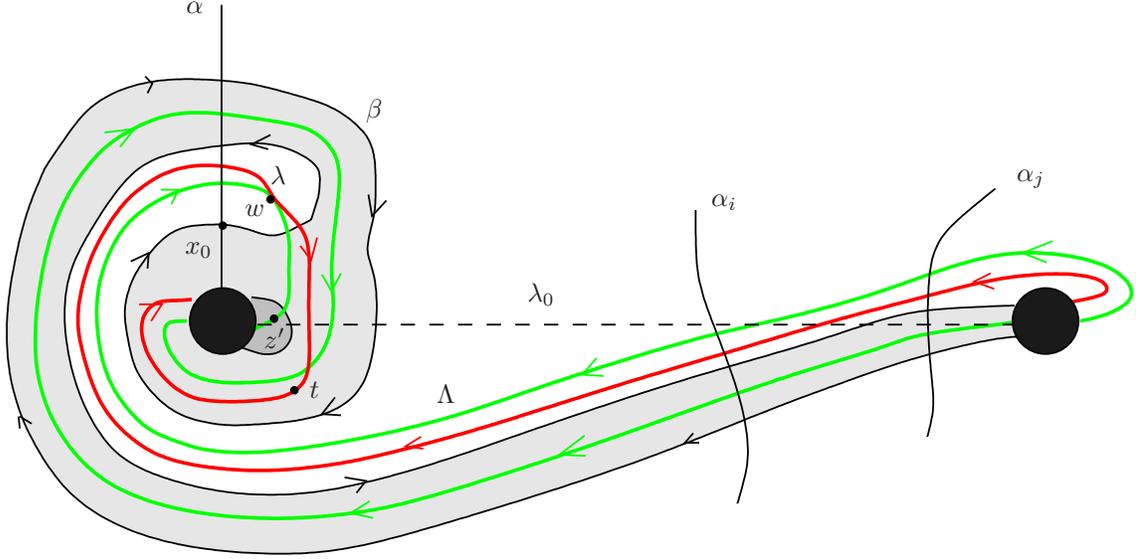}
\caption{A diagram for the $(P, Pn+1)$-cable of $K$.}
\label{cable-pic} 
\end{figure}

To construct a diagram for the $(P, Pn+1)$-cable of $K$, we first replace 
$\lambda$ by the $n$-framed longitude $\lambda=\lambda_0+n \beta$. Then, we perform a $P$-fold finger move of $\beta$ along $\lambda$, 
and replace the basepoint $z$ by $z'$ as in Figure \ref{cable-pic}. The diagram  $(\Sigma, \aalpha, \bbeta, w, z')$ now represents the cable $K'$.  The diagram also represents the original knot $K$ if we introduce another basepoint $t$, as in Figure \ref{cable-pic}. Notice that $\lambda$ can be now thought of as curve connecting $w$ to $t$ that intersects $\beta$ once and is disjoint from the 
other $\beta$'s. A longitude $\Lambda$ for the cable can be obtained by connecting $w$ and $z'$ in a similar fashion. 
We can rewrite (\ref{=0}) as 
\begin{equation} \label{la=0}
p \lambda  + (q -pn)\beta +  \sum_{i=0}^{g-1} r_i \alpha_i + \sum_{i=1}^{g-1} q_i \beta_i=0.
\end{equation}
This  relation gives rise to another  periodic domain $\P_n$ whose homology class equals $[F]$, and whose boundary  includes the new longitude $\lambda$.  We compute the multiplicities of this periodic domain as above.   If we pick the component $\D_0$ to be outside of the winding region (e.g.  the top right corner in  Figure \ref{cable-pic}) then it is clear  that the multiplicities of 
$\P_n$ are independent of $n$ outside of the winding region. Within the winding region, however, 
the multiplicities increase towards the center of the spiral formed by $\lambda$.  The finger move creates a number of parallel copies of $\beta$, and as one moves towards the center of the finger the multiplicities of $\P$ decrease. (An iterated finger, together with multiplicities of $\P$ in various regions, is shown in Figure 
\ref{iterated-finger}.)

\begin{figure}[htb] 
	\labellist
	\small\hair 2pt
	\pinlabel $0$ at 159 208
	\pinlabel $q-pn$ at 159 188
	\pinlabel $2(q-pn)$ at 159 170
	\pinlabel $3(q-pn)$ at 159 150
	\pinlabel $3(q-pn)+p$ at 159 130
	\pinlabel $2(q-pn)+p$ at 159 110
	\pinlabel $2(q-pn)+2p$ at 159 90
	\pinlabel $q-pn+2p$ at 159 72
	\pinlabel $q-pn+3p$ at 159 54
	\pinlabel $3p$ at 159 36
	\pinlabel $4p$ at 159 17 
	\endlabellist
\centering	
\includegraphics[scale=1.0]{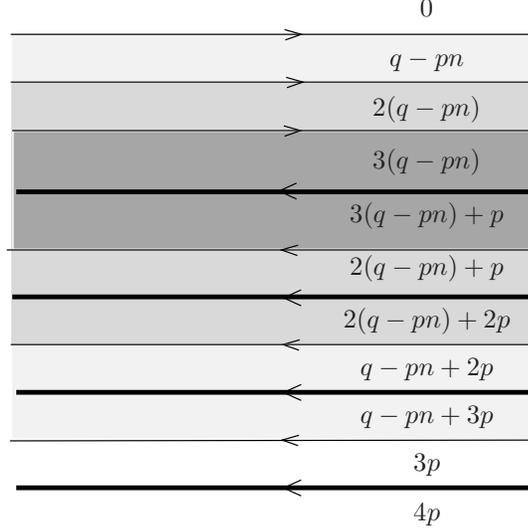}
\caption{A 3-fold finger and the multiplicities of  the periodic domain $\P_n$ given by equation \ref{la=0}.
 The thinner curve is the meridian $\beta$, the thicker curve is the longitude $\lambda$. When $n$ is large,
 the multiplicities inside the finger are smaller than the multiplicities outside, and  the multiplicities  increase as we move towards the center 
 of the $\lambda$-spiral in Figure \ref{cable-pic}.} 
\label{iterated-finger} 
\end{figure}

These considerations show that of the intersection points of $\alpha$ and $\beta$, 
the point $x_0$ (shown in Figure \ref{cable-pic}) has the highest multiplicity, and that  
the generators of $\CFK(Y, K)$ with the 
highest multiplicities are given by the set $x_0 \times C$, where $C$ is the set of $(g-1)$-tuples of intersection points of 
$\alpha_1, \dots \alpha_{g-1}$ and $\beta_1, \dots, \beta_{g-1}$ that have the highest multiplicities among all such $(g-1)$-tuples. 

To understand the Alexander gradings for the cable, $K'$, we must find a \relative  periodic domain representing $F'$ in the same diagram.  We now turn our attention to this task.

 Consider the cable $K'$ and its longitude $\Lambda$. The curve $\Lambda$ is homologous to $P(\lambda_0+n \beta) +\beta= P \lambda +\beta$. 
 We  have the null-homology
 $$
 p' \Lambda + (Rq-Rp-p') \beta + R(\sum_{i=0}^{g-1} r_i \alpha_i + \sum_{i=1}^{g-1} q_i \beta_i)=0,
 $$ 
which gives the rational periodic domain $\P'$ for $K'$.  It is clear that
that $[\P']=[F']\in H_2(Y\setminus \nbdKprime, \d (Y\setminus \nbdKprime))$.   

Now we look for generators with highest multiplicities with respect to $\P'$. As before, outside of the winding region these multiplicities are 
independent of $n$.  Moreover, we have
 \begin{equation} \label{Ptimes}
 \bar{n}_{x}(\P') = R \bar{n}_{x}(\P).
 \end{equation}
 
 For the intersection points with one coordinate on $\beta$, the above relation no longer holds, but the multiplicities of $\P'$ behave 
 similarly to the multiplicities of $\P$, increasing towards the center of the winding region and decreasing towards the center of the finger.
It follows that the top filtration level of $\CFK(Y, K')$ is given by  $x_0 \times C'$, where $C'$ is defined, analogous to $C$,
 as the set of $(g-1)$-tuples with the highest multiplicity. 
 Moreover, equation (\ref{Ptimes}) shows that the set $C$ is {\em identical} to the set $C'$. 
  This identifies the generators in the top filtration levels 
 of $\CFK(Y, [F], K)$ and $\CFK(Y, [F'], K')$. To identify the homologies in the top grading level, observe that  the differentials 
 on  $\CFK(Y, [F], K, \mathrm{top})=\CFK(Y, [F'], K', \mathrm{top})$ must both count holomorphic Whitney disks with $n_w=n_t=n_{z'}=0$ (see \cite[Proof of Lemma 3.6]{He}) thus the chain complexes 
  $(\CFK(Y, [F], K, \mathrm{top}), \d)$ and $(\CFK(Y, [F'], K', \mathrm{top}), d)$ are the same.

\end{proof}

\subsection{Proof of Theorem \ref{cont-inv}}
\noindent Recall the statement of  Theorem \ref{cont-inv}.

\begin{theorem} \label{hfkbottom}  Let $K \subset Y$ be a rationally fibered knot with rational fiber $F$, and let $c(\xi_K)$ denote the contact invariant of the contact structure $\xi_K$ induced by the associated rational open book decomposition.  Then $H_*(\F(-Y, [F],K, \mathrm{bottom}))\cong \FF\cm\langle c\rangle$.  Moreover,  $c(\xi_K)=  \iota_*(c)$, where $$ \iota: \F(-Y,[F],K,\mathrm{bottom})\rightarrow \CFa(-Y),$$ is the inclusion map of the subcomplex.  \end{theorem}

\begin{proof} Suppose $K$ has order $p$ in $H_1(Y)$. To establish the lemma, we will consider a cable $K'= K_{p,pn+1}$, with large $n>0$. 
Then $K'$ is a null-homologous fibered knot inducing an honest open book compatible with $\xi_{K}$ \cite[Theorem 1.8]{BEV}. (The page $F'$ of the open book for $K'$ 
can be constructed from $F$ by the procedure described before the statement of Proposition \ref{cable}, provided that we take a Thurston norm minimizing surface in $\nbdK\setminus \nbdKprime$ as the interpolating surface between  $K'$ and  $\partial F$).

Since $K'$ is null-homologous, the results of \cite{contOS} apply; thus   $H_*(\F(-Y, [F'],  K', \text{bottom}))\cong \FF\langle c\rangle$ and \begin{equation}\label{eq:well-defined1} c(\xi_{K'})=\iota'_*(c),\end{equation}  where  $\iota'$ is the inclusion map for the cable.  Moreover, since positive cabling doesn't change the contact structure, we have \begin{equation}\label{eq:well-defined2}c(\xi_K)=c(\xi_{K'}).\end{equation}

If we now reverse the orientation of the Heegaard surface in the proof of Proposition \ref{cable}, this has the effect of  changing the oriented manifold from $Y$ to $-Y$. It also has the effect of changing the sign of the multiplicities of the rational periodic domains. This reverses the Alexander grading (up to a translation), and proves  $$
\F(-Y, [F], K, \text{bottom})=\F(-Y, [F'],  K', \text{bottom}).
 $$ 
Let $c$ denote a generator of the homology of this complex.  Since the (singly-pointed) Heegaard diagram for  $-Y$ is independent of the additional basepoint used to specify $K$ or $K'$, we have \begin{equation}\label{eq:well-defined3}\iota_*(c)=\iota'_*(c).\end{equation} Indeed, the respective inclusion maps  can be obtained by taking a cycle representative for $c$ and considering the homology class it represents in $\HFa(-Y)$ by forgetting the respective additional basepoint.  Combining \eqref{eq:well-defined1}, \eqref{eq:well-defined2}, and \eqref{eq:well-defined3} yields the result.
 \end{proof}


\noindent{\bf Proof of Corollary \ref{cor:lens}:} Suppose that integer surgery on $K\subset L(p,q)$ is the $3$-sphere.  Then there is an induced knot $K'\subset S^3$ on which $\pm p$ surgery produces $L(p,q)$ (the core of the surgery torus).  In this situation, \cite[Theorem 1.1]{NiFibered} implies that $K'$ is fibered, and hence $K$ is rationally fibered.  By reflecting $K'$, if necessary, we may assume the surgery slope is $+p$ (this may change the orientation of $L(p,q)$, but as we ultimately consider both orientations on $L(p,q)$ this point will not affect the argument).  Now \cite[Theorem 1.4]{He4} states that either $p\ge 2g(K')$, in which case \begin{equation}\label{eq:equality}\mathrm{rk}\ \HFK(L(p,q),K)=\mathrm{rk} \ \HFa(L(p,q))=p,\end{equation} or $p=2g(K')-1$ in which case $$\mathrm{rk}\ \HFK(L(p,q),K)=\mathrm{rk}\ \HFa(L(p,q))+2=p+2.$$  The latter case, however, is ruled out by \cite[Theorem 1.2]{Greene}, and thus the rank of the knot Floer homology of $K$ is equal to the rank of the Floer homology of the manifold in which it sits.  This immediately implies that the inclusion 
$$
\iota: \F(L(p,q),K,\mathrm{bottom})\rightarrow \CFa(L(p,q))
$$ 
is injective on homology: the homology of $\F(L(p,q),K,\mathrm{bottom})$ is the bottom knot Floer homology group, which survives under the spectral sequence from the knot Floer homology of $K$ to  $\HFa(L(p,q))$ by the equality of ranks \eqref{eq:equality}.  Thus  $0\ne c(\xi_{\overline{K}})\in \HFa(-(-L(p,q)))$.  Since reversing the orientation of $L(p,q)$ changes neither the rank of the Floer homology of $K$ nor the ambient manifold, the inclusion $$ \iota: \F(-L(p,q),K,\mathrm{bottom})\rightarrow \CFa(-L(p,q)),$$ is also injective on homology, indicating that the contact structure $\xi_{{K}}$ induced by $K$ on $L(p,q)$ also has non-vanishing invariant.  \qed 

\begin{remark} The corollary is somewhat more general.  Indeed, let $K\subset Y$ be a knot in an integer homology sphere L-space whose complement is irreducible, and let $K_n\subset Y_n$ be the induced knot.  Then if $Y_n$ is an L-space and $n\ge 2g(K)$, the conclusion holds; that is, $K_n$ is rationally fibered and induces a tight contact structure on both $Y_n$ and $-Y_n$.
\end{remark}

\subsection{A Heegaard diagram for rationally fibered knots}
We can mimic the construction in \cite{contOS}
to pinpoint $c(\xi)$ as the homology class of a specific generator in a particular Heegaard diagram constructed from the open book.

\begin{prop} \label{intersection-pt} Let $K \subset Y$ be a rationally fibered knot. There is a Heegaard diagram 
adapted to $(Y, K)$ so that $\F(-Y, K, [F], \mathrm{bottom})$ is generated by a single intersection point $\bf{c}\in \Ta\cap \Tb$. Thus $c(\xi_K)=[{\bf c}]\in \HFa(-Y)$.
\end{prop}

\begin{proof} We  adapt  \cite[Theorem 1.1]{contOS} to construct the required Heegaard diagram.  
Since $K$ is fibered, the complement of $K$ has a Dehn filling $Y_0$ which fibers over $S^1$.
We first construct a Heegaard diagram for $Y_0$, and then recover the desired diagram for $Y$ by a rational surgery.   

\begin{figure}[htb] 
	\labellist
	\small\hair 2pt
	\pinlabel $\beta_1$ at 133 165
	\pinlabel $\alpha_1$ at 133 -5
		\pinlabel $\beta_2$ at 28 48
	\pinlabel $\alpha_2$ at 30 111
	\pinlabel $\gamma$ at 152 98
	\pinlabel $\mu$ at 108 107
		\pinlabel $A$  at 12 80
	\pinlabel $\bar A$  at 250 80
	\endlabellist
\centering	
\includegraphics[scale=1.4]{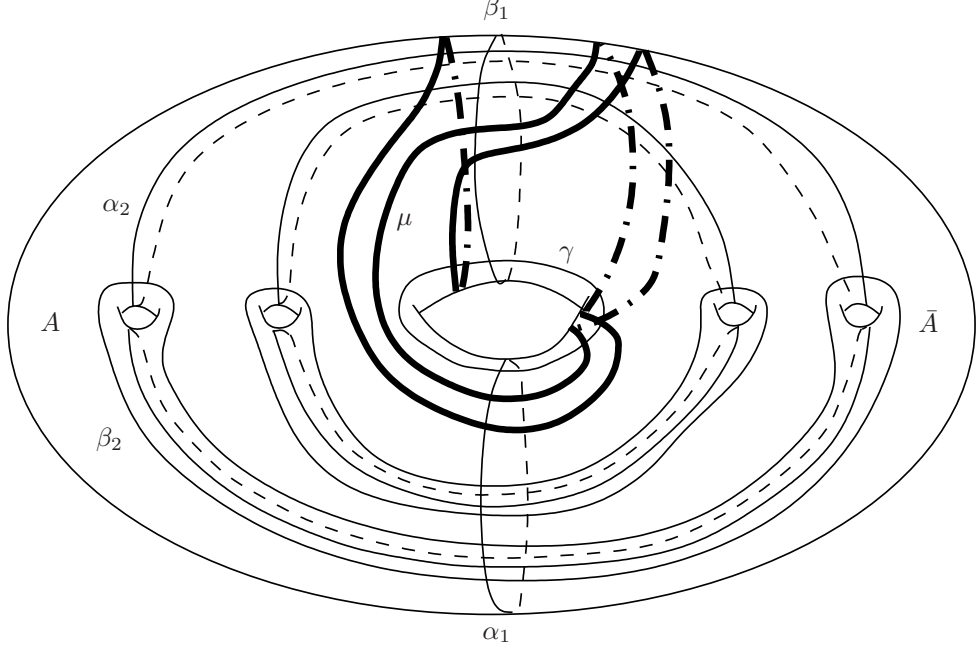}
\caption{Constructing a Heegaard diagram from rational open book. The figure shows the Heegaard diagram arising form the open book with trivial monodromy. (To avoid overloading the picture, we have not drawn some of the $\alpha$- curves here. The missing curves lie in the back of the surface,
in the top part of the diagram.) For a general open book, the $\beta$- curves in the $\bar A$ region 
will be affected by the action of the open book monodromy (this is not shown in the figure).} 
\label{OS-diag} 
\end{figure}

Let $F$ be the rational 
Seifert surface for $K$; capping it off, we obtain a closed surface $\hat{F}$ of genus $g$. We first follow the procedure from \cite{contOS} to obtain 
the Heegaard diagram for $S^1 \times \hat{F}$. Start with a genus $g$ surface $A$ with two boundary components, $\alpha_1$ and $\beta_1$. 
Let $\xi_2, \dots, \xi_{2g+1}$ and $\eta_2, \dots, \eta_{2g+1}$ be two $2g$-tuples of pairwise disjoint arcs in $A$ such that $\xi_i$ meets $\eta_i$
at a single point of transverse intersection, denoted $c_i$, and $\xi_i \cap \eta_j = \emptyset$ for $i\neq j$.   A Heegaard diagram 
$(\Sigma, \{\alpha_1, \alpha_2, \dots \alpha_{2g+1}\}, \{\beta_1, \beta_2, \dots \beta_{2g+1}\})$ can then be obtained by doubling $A$ along its boundary; that is, we consider the surface $\bar{A}$ obtained by reflecting $A$ across its boundary, and glue $A$ and $\bar{A}$ together to form a closed surface $\Sigma$.  This gluing produces  closed  curves $\alpha_i$ resp.
$\beta_i$, $i=2, \dots {2g+1}$  by gluing $\xi_i$ to its copy $\bar{\xi}_i$, resp. $\eta_i$ to its copy $\bar{\eta}_i$.
The result is a Heegaard diagram for $S^1 \times \hat{F}$. Moreover,  removing $\beta_1$ results in a Heegaard diagram for the complement 
of   $S^1\times \{\mathrm{pt}\}\subset S^1 \times \hat{F}$. This manifold is homeomorphic to the complement of the knot 
$B \subset \#^{2g} S^1 \times S^2$, where $B$  is the binding for the open book with trivial monodromy.   
The meridian of $B$ is represented by the curve $\gamma = \delta \cup \bar{\delta} \subset \Sigma$, 
formed by doubling an arc    $\delta \subset A$  connecting $\beta_1$ and $\alpha_1$.
These diagrams will be admissible 
after additional isotopies (finger moves) of the attaching circles \cite{contOS}.

To obtain a Heegaard diagram for $Y_0$, we must change the monodromy of the fibration. The monodromy map for $Y_0$ is the extension 
to $\hat{F}$ of an automorphism $\phi:F \to F$. Thinking of $F$ as the complement of $\bar{\delta}$ in $\bar{A}$, we extend it by the identity 
to get an automorphism  $\Phi: \Sigma \to \Sigma$. The diagram $(\Sigma, \aalpha, \bbeta=\{\beta_1, \Phi(\beta_2), \dots, \Phi(\beta_{2g+1})\})$  
represents $Y_0$,  and $(\Sigma, \aalpha, \{\Phi(\beta_2), \dots, \Phi(\beta_{2g+1})\})$ represents the complement of a fibered knot $\ti{K}$. 
With finger moves, these diagrams can be made weakly admissible for all $\Spinc$ structures, as above.

Since $Y$ is a Dehn filling of the complement of $\ti{K} \subset Y_0$, we can  obtain 
a Heegaard diagram for $Y$ by replacing  $\beta_1$ by the meridian $\mu$ of $K \subset Y$.
If $Y$ is obtained by a $p/q$-surgery on $\ti{K}$ (with respect to the longitude given by $\gamma$), 
then $\mu$ can be represented by a curve on  $\Sigma$ homologous to $p \beta_1 + q \gamma  $.   A longitude for knot $K$ 
is now given by a curve $\lambda$ on $\Sigma$ that intersects $\mu$ transversely at a single point, and is disjoint from the curves $\beta_2, \dots \beta_{2g+1}$.  Such a longitude is homologous to $b \beta_1 + a \gamma$, for $a,b$ satisfying $pa-qb=-1$.  We may assume that, like $\mu$, the curve $\lambda$ is supported in a small neighborhood of $\beta_1 \cup \gamma$.     

The resulting Heegaard diagram is shown in Figure \ref{OS-diag}, and Figure \ref{zoom-in} provides a closer look at the region containing $\mu$, $\beta_1$ and $\gamma$.

\begin{figure}[htb] 
	\labellist
	\small\hair 2pt
	\pinlabel $\alpha_1$ at 90 144
		\pinlabel $z$ at 63 105
	\pinlabel $w$ at 73 140
	\pinlabel $\lambda$ at 107 127
	\pinlabel $\mu$ at 125 145
	\endlabellist
\centering	
\includegraphics[scale=1.0]{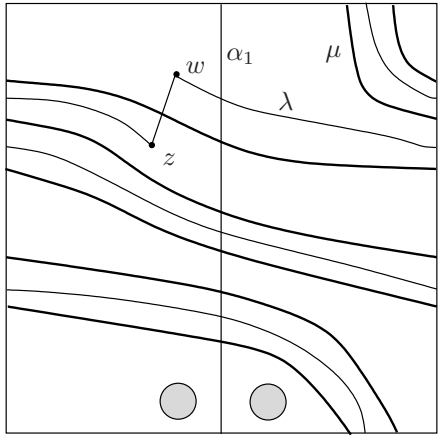}
\caption{Zooming in on Figure \ref{OS-diag}: the meridian and the longitude of the binding of rational open book.}
\label{zoom-in} 
\end{figure}

Observe that the Heegaard surface $\Sigma$ is cut by the attaching circles 
into a large region lying in $A$, a large region lying in $\bar{A}$, a number of regions with boundary on  $\lambda$, $\mu$, and $\alpha_1$, 
(see Figure \ref{zoom-in}), and a number of small regions in $\bar{A}$.  

Further observe that there is a 2-chain in $\Sigma$ whose boundary is $\alpha_1+\beta_1$. Since $\beta_1= q \lambda - a \mu$, 
we can find a \relative periodic domain $\P$ whose homology class equals that of the fiber, with $\d\P=  \alpha_1 + q\lambda-a\mu$.
The multiplicities of $\P$ are 0 in the large region in $A$, 1 in the large region in $\bar{A}$. 
The multiplicities in the regions in Figure \ref{zoom-in} require a bit more work, 
but are also straightforward to compute. To find them, we start with 0 in the top left corner of Figure \ref{zoom-in}, and then move to the neighboring regions, 
changing the multiplicity by $\pm 1$ when we cross $\alpha_1$, by $\pm q$ when we cross $\lambda$ and by $\pm a$ when we cross $\mu$.
(The signs depend on the direction in which the curves are traversed. In Figure \ref{zoom-in}, if travel upwards, the multiplicity of $\P$ increases 
when we cross $\lambda$, decreases when we cross $\mu$. When crossing $\alpha_1$, the multiplicity increases by 1 from left to right.) 

Our goal  now is to show that there is 
a unique generator $\x$ which minimizes the multiplicity $n_{\x}(\P)$. Since $\alpha_1$ is disjoint from all the curves $\beta'_2, \dots \beta'_{2g+1}$, 
every generator has the form $\x=(x, x_2, \dots x_{2g+1})$, where $x\in \alpha_1 \cap \mu$. As in \cite{contOS}, a generator minimizing $n_{\x}(\P)$ will  have $x_2, \dots x_{2g+1}$ contained in $A$. Thus, we must have $\{x_2, \dots, x_{2g+1}\}=\{c_2, \dots c_{2g+1}\}$; in particular, 
these coordinates of $\x$ are uniquely determined. The lowest value of  $n_{\x}(\P)$ will be attained by those  generators $\x=(x, c_2, \dots c_{2g+1})$
for which $n_x(\P)$ is the lowest among all $x\in \alpha_1 \cap \mu$.

To complete the 
proof of the proposition,  we  show that the values of $n_x(\P)$ are mutually distinct for the various points $x\in \alpha_1 \cap \mu$.  
If $n_x(\P)=n_{x'}(\P)$, then the multiplicities at the four corners of $x$ and $x'$ would 
be the same, since the multiplicities in the corners around $x$ and $x'$ change in the same way when the curves $\alpha_1$ and $\mu$ 
are crossed.
 Consider, however, the shortest path from $x$ to $x'$ along $\alpha_1$.  If we cross the curve $\lambda$ a total of $r_{\lambda}$ times and the curve $\mu$ a total of $r_{\mu}$ times along this path, then we have $r_{\lambda} q - r_{\mu} a=0$. However, since $\lambda$ intersects $\alpha_1$ at $a$ points 
and $\mu$ intersects $\alpha_1$ at $q$ points, $0<r_{\lambda}<a$ and $0<r_{\mu}<q$. Thus $r_{\lambda} q - r_{\mu} a=0$   contradicts the fact that  $\mathrm{gcd}(a, q)=1$. 
This shows that there is a unique 
point $c\in \alpha_1 \cap \mu$ with  smallest $n_c(\P)$.    

\end{proof}

\begin{remark} Theorem \ref{hfkbottom} and Proposition \ref{intersection-pt} provide two independent proofs of the fact that a rationally null-homologous fibered knot has knot Floer homology of rank 1 in the extremal Alexander grading. (This extends the analogous result for null-homologous knots, \cite{contOS}.) Yet another proof can be obtained by the sutured Floer homology of \cite{Ju}. \end{remark}

\section{The contact invariant of rational open books induced by surgery}  
 
 In this section we prove our non-vanishing theorem for the contact invariant of the  contact structure induced by the rational open book which results from surgery on the binding of an honest open book.  More precisely, recall that if we perform surgery on the binding of an honest open book, then the core of the surgery torus is a knot in the new manifold whose complement fibers over the circle (as it is homeomorphic to the complement of the original binding).  Theorem \ref{p/q-surgery}  says that if the contact invariant associated to the original open book is non-zero, then the contact invariant of the induced rational open book  is also non-zero, provided that the surgery parameter is sufficiently large.  

\vspace{0.1in} 
 
\noindent{\bf Theorem \ref{p/q-surgery}}{\em \ \  Let $K\subset Y$ be a fibered knot with genus $g$ fiber, and $\xi$ the contact structure induced by the associated open book.  Let $K_{p/q}\subset Y_{p/q}$ be the rationally fibered knot arising   as the core of the solid torus used to construct $p/q$ surgery on $K$, and $\xi_{p/q}$ the contact  structure

 \vspace{-0.035in} \noindent  induced by the associated rational open book.  Suppose  $c(\xi)\! \! \in\!  \HF(-Y)$ is non-zero. Then $c(\xi_{p/q}) \in \HF(-Y_{p/q})$ is  non-zero for all $p/q \geq 2g$.    
}

\vspace{0.1in}

We prove the theorem in steps, each step expanding the range of slopes for which the theorem holds.   The first step is to show that the theorem holds for all sufficiently large integral slopes.  This is accomplished by Theorem \ref{thm:nonvanishing} below.  The key tool in this step  is an understanding of the relationship between the knot Floer homology of a knot $K\subset Y$ and the knot Floer homology of the core of the surgery torus $\coren\subset Y_n$.   This relationship was studied in \cite{He2}, following the ideas of \cite{knotOS}.   We begin  this section with a detailed discussion of these results, and prove a generalization (Theorem \ref{thm:core2}) which will serve as the cornerstone of our proof.

Our next step is to establish the theorem for all integral slopes $n\ge 2g$.  We accomplish this with  Theorem \ref{2g}, whose proof relies on a surgery exact sequence for the knot Floer homology of the core, together with an adjunction inequality.  

Finally, we extend our results to all rational slopes $p/q \ge 2g$.  This argument is geometric, showing that the contact structures $\xi_{p/q}$ with rational slope can be obtained from those with integral slope by a sequence of Legendrian surgeries.

\subsection{The knot Floer homology of the core of the surgery torus}

 We begin by stating a slightly rephrased version of \cite[Theorem 4.1]{He2}. We use notation of \cite{knotOS}. 
 
\begin{theorem}\label{thm:core1}  Let $K\subset Y$ be a null-homologous knot, and let $\Yn$ denote the manifold obtained by $n$-framed surgery on $K$.  Then for all $n\gg0$ sufficiently large, we have $$\CFa(\Yn,\spinc_m)\simeq C\{\max(i, j-m)=0\},$$
where $C\{\max(i, j-m)=0\}$ denotes the subquotient complex of $CFK^{\infty}(Y,K)$ whose $(i,j)$ filtration levels satisfy the stated constraint.

  Moreover, the core of the surgery torus induces a knot $\coren\subset \Yn$ and the filtration of $\CFa(\Yn,\spinc_m)$ induced by $\coren$ is filtered chain homotopy equivalent to the two-step filtration:
$$
0 \subset C\{i<0, j=m\} \subset C\{\max(i, j-m)=0\}. \qed 
$$ 
\end{theorem}
The first part of the theorem is simply \cite[Theorem 4.4]{knotOS}.  The second part, which deals with the filtration induced by $\coren$, was stated for $Y=S^3$ in the form above in \cite[proof of Theorem 4.1, pg 129]{He2}.  The proof carries through verbatim to general $Y$.  We also note that the core of the surgery torus is isotopic to the meridian of $K$, viewed as a knot in $\Yn$.  The original statement was phrased in these terms.

Since  the contact invariant associated to $\coren$ is calculated using the bottom subcomplex in the  knot Floer homology filtration of $\coren$, we need to understand what ``bottom" means in the theorem above.   Thinking of the filtration as a filtration by relative $\SpinC$ structures on $\Yn\setminus \nbd\coren$,  the theorem above determines this difference in the case of relative $\SpinC$ structures that project to the same absolute $\SpinC$-structure on $\Yn$ under the natural filling map \eqref{eq:filling}.


Thus we need to understand the difference between the relative $\SpinC$ structures (or, if the reader prefers, the Alexander grading difference) associated to knot Floer homology groups for the varying $\spinc_m\in \SpinC(Y_n)$.   Since the difference of two relative $\SpinC$ structures lies in $H^2(\Yn\setminus \nbd\coren,\partial (\Yn\setminus \nbd\coren))\cong H^2(Y)\oplus \Z$, we should make a few remarks about the algebraic topology of this situation.  The first is to remind the reader that $\Yn\setminus \nbd\coren\cong Y\setminus \nbdK$, so the algebraic topology is, in a sense,  identical.  The key conceptual difference is that we have changed the natural framing on the boundary of this manifold.  Thus, while $\mu_K$ generates the additional $\Z$ factor in $H_1(Y\setminus \nbd K)\cong H_1(Y)\oplus \Z$, the meridian of ${\coren}$ does not generate the $\Z$ factor in  $H_1(Y_n\setminus \nbd \coren)\cong H_1(Y)\oplus \Z$.  Indeed, $[\mu_{\coren}]=n\cm \rho$ for a class generating this summand, and it is easy to see that $\rho=[\coren]$, the homology class of a push-off of $\coren$ into its complement.

 Before stating the refined version of Theorem  \ref{thm:core1} we establish some notation.  Let $$S_m=C\{i<0, j=m\} $$
$$ Q_m= C\{i=0,j\le m\}$$
be the sub and quotient complexes in the filtration of $C\{\max(i, j-m)=0\}$ given by the theorem.  The direct sum of all the knot Floer homology groups of $\coren$ (without the Alexander grading) is then given by $$\HFK(\Yn,\coren)=\bigoplus_{m=-\lfloor n/2 \rfloor+1 }^{\lfloor n/2\rfloor}  H_*(S_m)\oplus H_*(Q_m).$$
A complete description of the knot Floer homology of the core of the surgery is given by

\begin{theorem}\label{thm:core2}
Let $K\subset Y$ be a null-homologous knot and $\coren\subset \Yn$ be the core of the surgery torus, viewed as a knot in the manifold obtained by $n$-framed surgery on $K$. Then for all $n\gg0$ sufficiently large,  the Alexander grading difference between the various knot Floer homology groups is given by 
$$
\begin{array}{lll} \gr( S_m)-\gr(Q_m)&= &n\\
\gr( S_i)-\gr (S_j)&= &-(i-j)\\
\gr( Q_i)-\gr(Q_j)&= & -(i-j)
\end{array}
$$
for all $-\lfloor n/2 \rfloor+1 \le m,i,j \le \lfloor n/2\rfloor$.
\end{theorem}
\begin{remark} The filtration on $\CFa(\Yn)$ induced by $\coren$ is most easily understood graphically.  For this we refer the reader to Figure \ref{fig:hooks}.
\end{remark}

\begin{figure}
\labellist
	\small\hair 2pt
	\pinlabel $i=0$ at 175 16
	\pinlabel $Q_1$ at 174 43
	\pinlabel $S_0$ at 11 171
	\pinlabel $S_1$ at 11 204
	\pinlabel $S_2$ at 11 235
	\pinlabel $j=0$ at 240 171
	
	\endlabellist
\centering
\includegraphics[scale=1.0]{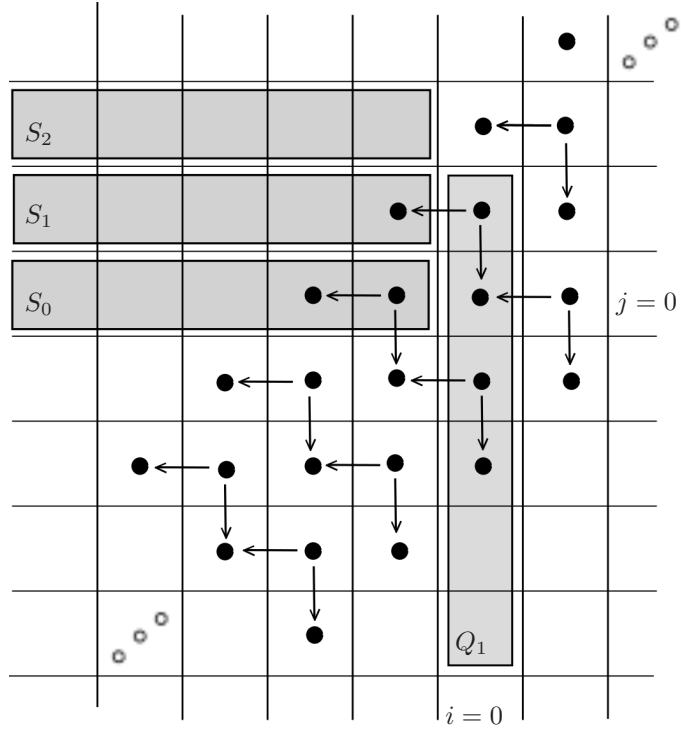}
\caption{\label{fig:hooks}   Shown is $CFK^{\infty}(S^3,K)$, for $K$ the $(2,5)$ torus knot.  Dots equal $\FF$, and arrows are non-trivial terms in the differential. The $\Z\oplus \Z$ filtration is given by the $(i,j)$ coordinates. The non-trivial knot Floer homology group for $\coren\subset S^3_n(K)$ with lowest Alexander grading is the homology of the subcomplex $S_1$ (while $S_2$ has lower Alexander grading, its homology is trivial).  The homology of $S_0$ is the knot Floer homology group with Alexander grading $1$ greater than that of $S_1$.  The homology of $Q_1$ is the knot Floer homology group in Alexander grading $n$ greater than $S_1$.}

\end{figure}

\begin{figure}
	\begin{center}
\psfrag{b}{$\beta_g$}
\psfrag{a}{$\alpha_g$}
\psfrag{c}{$\gamma_g$}
\psfrag{theta}{$\Theta$}
\psfrag{w}{$w$}
\psfrag{z}{$z$}
\psfrag{y}{$y$}
\psfrag{z'}{$z'$}
\psfrag{x-1}{$x_1$}
\psfrag{x-2}{$x_2$}
\psfrag{x-3}{$x_3$}
\psfrag{x0}{$x_0$}
\psfrag{x1}{$x_{-1}$}
\psfrag{x2}{$x_{-2}$}
\psfrag{x3}{$x_{-3}$}
\psfrag{0}{$0$}
\psfrag{1}{$1$}
\psfrag{2}{$2$}
\psfrag{psi1}{$\psi_1$}
\psfrag{psi2}{$\psi_{-2}$}
\psfrag{alpha1}{$\alpha_1$}
\includegraphics[width=3.5in]{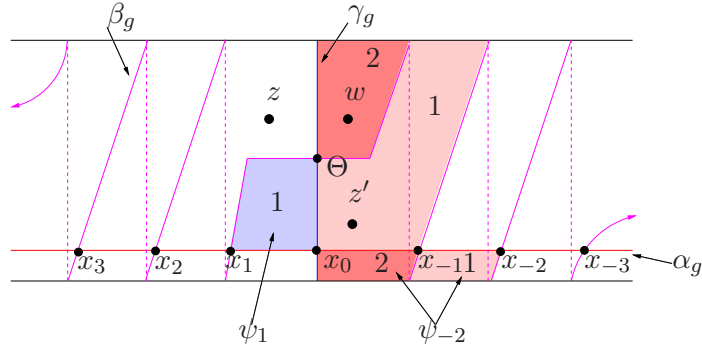}
\caption{\label{fig:winding}
The ``winding region" of the Heegaard triple diagram.   A small triangle $\psi_1$ connects $x_1$ to $x_0$ and a small triangle $\psi_{-2}$ connects $x_{-2}$ to $x_0$. 
} 
\end{center}
\end{figure}

\begin{proof}  The proof is a straightforward extension of the proof of \cite[Theorem 4.1]{He2} which, in turn, was an extension of the proof of \cite[Theorem 4.4]{knotOS}.  Both proofs were local, and involve an examination of the winding region in a Heegaard triple diagram representing the 2-handle cobordism from $\Yn$ to $Y$.  See Figure \ref{fig:winding} for a depiction of this region.  

Given this Heegaard triple diagram, a chain map  $$\CF(\Yn,\spinc_m)\longrightarrow C\{\mathrm{max}(i,j-m)=0\}$$ is defined in \cite{knotOS} by counting pseudo-holomorphic triangles. The obvious small triangles present in the winding region (together with their extensions to $g$-tuples of small triangles in the rest of the triple diagram) induce a bijection of groups, provided that $n$ is large enough to ensure that all the intersection points for $\spinc_m$ have $\alpha_g$ component in the winding region.   Moreover, these small triangles constitute the lowest order terms of the chain map with respect to the area filtration, and this latter fact shows that the chain map induces an isomorphism on homology.

To understand the filtration of $\CF(\Yn,\spinc_m)$ induced by $\coren$, we observe that the placement of a third basepoint $z'$ on the Heegaard triple diagram has the property that $(\Sigma,\alphas,\betas,z',w)$ represents $\coren$.  The bijection induced by small triangles from the last paragraph is such that \begin{enumerate} \item if an intersection point for $\CF(\Yn)$  has $\alpha_g\cap \beta_g$ component  lying to the right of $x_0\in \alpha_g\cap \gamma_g$, then it is sent to a subcomplex $S_m$, with the distance to $x_0$ proportional to $-m$, \item if the 
$\alpha_g\cap \beta_g$ component  is to the left of $x_0\in \alpha_g\cap \gamma_g$, the intersection point is sent to a quotient $Q_m$, with the distance to $x_0$ proportional to $-m$. \end{enumerate} Finally, any two intersection points $\x,\y$ representing $\spinc_m$ can be connected by a Whitney disk $\phi$ which satisfies:
$$  n_{z'}(\phi)-n_w(\phi)=  \pm1 $$
if the $\alpha_g\cap \beta_g$ components of $\x$ and $\y$ are on opposite sides of $x_0$, and 
$$  n_{z'}(\phi)-n_w(\phi)=  0$$
otherwise.  Since $$(n_{z'}(\phi)-n_w(\phi))\cm[\mu_{\coren}]= \epsilon(\x,\y)\in H_1(\Yn\setminus\nbd \coren)\cong H_1(Y)\oplus \Z\langle[\coren]\rangle,$$ and
$$\langle \PD([\mu_{\coren}]),[F,\partial F]\rangle= n[\coren]\cm[F,\partial F]=n,$$
this proves Theorem \ref{thm:core1} (and the first part of the present generalization).  

To complete the theorem, we must understand the filtration difference between  the subcomplexes $S_i,\ S_j$ (respectively, the quotient complexes $Q_i, \ Q_j$) with $i\ne j$.  By the transitivity of the filtration, it will suffice to understand the difference between $S_i$ and $S_{i+1}$. Consider a generator $\x=\{x_{-l},\mathbf{s}\}$ lying in the subcomplex $S_i$, where $x_{-l}\in \alpha_g\cap\beta_g$ and $\mathbf{s}$ is the remaining $(g-1)$-tuple of intersection points.    There is a corresponding generator $\x'=\{x_{-l+1},\mathbf{s}\}$ which lies in $S_{i+1}$, according to $(1)$ above.  These two generators can be connected by a curve which wraps once around the neck of the winding region; that is, $\epsilon(\x,\x')=[\coren]$,  since this curve represents the generator of $H_1(\Yn\setminus \coren)$. Thus we have 
$$ \gr(\x)-\gr(\x')=   \langle \PD([\coren]) ,[F,\partial F]\rangle = 1.$$
This proves the second line in the theorem.  The third is given by a mirror argument on the left side of $x_0$.

\end{proof}

\subsection{Non-vanishing for sufficiently large integral slopes}

With a firm understanding of the relationship between the knot Floer homology of $K\subset Y$ and $\coren\subset Y_n$, we can easily establish a non-vanishing theorem for sufficiently large integral surgeries.

\begin{theorem}\label{thm:nonvanishing} Suppose that the contact structure $\xi$, compatible with an open book $(Y, K)$, has $c(\xi)\neq 0$. 
For $n>0$, perform $n$-surgery on $K$, and consider the induced rational open book $(Y_n, \coren)$ and the compatible contact structure 
$\xi_n$. Then $c(\xi_n)\neq 0$ if $n$ is sufficiently large. 
\end{theorem}

\begin{proof} Let $\xi$ be a contact structure compatible with an open book associated to a fibered knot $K\subset Y$, and let $\xi_n$ be the contact structure compatible with the rational open book associated to $\coren\subset Y_n$.  By definition, the \os \ contact element $c(\xi)$ is the image in $\HF(-Y)$ 
of the generator of $H_*(\F(-Y, K,\mathrm{bottom}))\cong 
\mathbb{F}$, under the map induced by the inclusion:
$$\iota: \F(-Y, K,\mathrm{bottom}) \hookrightarrow \CF(-Y).$$  By Theorem \ref{hfkbottom}, this definition extends to rational open books.  That is, the contact element $c(\xi_n)$  is equal to the image in $\HF(-Y_n)$ of the generator 
of $\F(-Y_n, \coren, \mathrm{bottom})$ under the corresponding inclusion
$${\iotan}: \F(-Y_n, \coren,\mathrm{bottom}) \hookrightarrow \CF(-Y_n).$$ To prove the theorem, we need only understand the relationship between the inclusion maps $\iota,{\iotan}$, as governed by Theorem \ref{thm:core2}.  Indeed, the theorem follows immediately from

\bigskip
\noindent{\bf Claim:} Let $n$ be sufficiently large so that Theorem \ref{thm:core2} holds.  Then
$$\iota_* \ne 0\ \   \Longleftrightarrow \ \ \iotan_*\ne 0.$$

To prove the claim, we first translate it into a statement about the topmost knot Floer homology group by a duality theorem.    Consider the short exact sequence $$0\hookrightarrow \F(Y,K,\mathrm{top}-1)\hookrightarrow \CF(Y)\rightarrow \CFK(Y,K,\mathrm{top})\rightarrow 0,$$
and the associated connecting homomorphism 
$$  \HFK(Y,K,\mathrm{top})\overset{\delta_*}\longrightarrow H_*( \F(Y,K,\mathrm{top}-1)).$$
A duality theorem \cite[Proposition 2.5]{propOS} states that the Floer homology of $-Y$ is the Floer cohomology of $Y$.  The knot $K$ can be viewed in $-Y$, and there is a corresponding duality theorem for the filtrations \cite[Proposition 3.7]{knotOS} (see also \cite[Proposition 15]{He3} for the formulation we use here).  An immediate consequence of this duality, since the rank  $\HFK(Y,K,\mathrm{top})$
is 1,  is that $$ \iota_* \ne 0  \ \ \Longleftrightarrow \ \ \mathrm{ker}\ \delta_* \ne 0$$ 
The duality theorem holds for rationally null-homologous knots, and thus the claim reduces to showing that $$\mathrm{ker}\ \delta_*\ne 0 \ \ \Longleftrightarrow \ \ \mathrm{ker}\ \deltan_* \ne 0.$$
To do this, observe that Theorem \ref{thm:core2} shows that  $\HFK(Y_n,\coren,\mathrm{top})\cong H_*(Q_{-g})$, the homology of the $-g$-th quotient in the notation of that theorem, where $g=g(K)$ is the minimal genus of any embedded surface in the same homology class as $F$ (to see this, observe that $\Filt(Q_i)>\Filt(S_j)$ for all $i,j$, and that $H(Q_i)=0$ for all $i<-g(K)$, by the adjunction inequality \cite[Theorem 5.1]{knotOS}).  The map 
$$\deltan_*: \HFK_*(Y_n,\coren,\mathrm{top})\cong H_*(Q_{-g})\longrightarrow H_{*-1}(\F(Y_n,\coren,\mathrm{top}-1))$$ factors through the map induced by inclusion $S_{-g}\hookrightarrow \F(Y_n,\coren,\mathrm{top}-1)$.  Again, this follows from Theorem \ref{thm:core2}, as there are simply no generators in any other filtration levels which could be connected to those in $H_*(Q_{-g})$ by Whitney disks.  Thus  $\mathrm{ker} \ \deltan_*\ne 0$ if and only if     
$$H_*(Q_{-g})\longrightarrow H_{*-1}(S_{-g})$$ has non-trivial kernel or, equivalently,  if $$H_*(C\{i=0,j=-g\})\longrightarrow H_{*-1}(C\{i<0,j=-g\})$$ has non-trivial kernel.   But this last map is the same, as a relatively graded map, as  $$H_*(C\{i=0,j=g\})\longrightarrow H_{*-1}(C\{i=0,j<g\})$$ by the filtered chain homotopy equivalence between the vertical and horizontal complexes $C\{j=-g\}$ and $C\{i=0\}$, respectively \cite[Proposition 3.8]{knotOS}.  This last map, however, is $\delta_*$.

This completes the proof of the claim, and hence of Theorem \ref{thm:nonvanishing}.

\end{proof}

\subsection{Non-vanishing for integral slopes $n\ge 2g$}

\begin{theorem}\label{2g}
Suppose that the contact structure $\xi$, compatible with an open book $(Y, K)$ of genus $g$, has $c(\xi)\neq 0$. 
Then for all $n\ge2g$ the contact structure 
$\xi_n$ compatible with the induced rational open book $(Y_n,  K_n)$ has $c(\xi_n)\neq 0$.
\end{theorem}
\begin{proof}  Perhaps the most aesthetically appealing proof would be to show that Theorem \ref{thm:core2} holds for all $n\ge 2g$, regardless of the knot.   We will take the easier route, and content ourselves to prove what is necessary for our application.

The  proof makes use of a surgery exact sequence, together with an adjunction inequality.   Recall the integer surgeries long exact sequence for the Floer homology of closed manifolds which differ by surgery along a null-homologous knot $K\subset Y$ \cite[Theorem 9.19]{propOS}:

$$...\rightarrow \bigoplus_{i=1}^{n}\HFa(Y)\rightarrow \HF(\Yzero)\rightarrow \HF(\Yn)\rightarrow...\ .$$ This sequence holds for any framing $n> 0$.   Moreover, the sequence decomposes as a direct sum of $n$ exact sequences corresponding to the $\Z/n\Z$ factor in $H^2(\Yn)\cong H^2(Y)\oplus \Z/n\Z$, 

$$...\rightarrow \HFa(Y)\longrightarrow \!\!\!\! \underset{\{\spinc |  \langle c_1(\spinc),[\widehat{F}]\rangle =2m\ \mathrm{mod}\ 2n \}}\bigoplus \!\!\!\! \HF(\Yzero,\spinc)\longrightarrow \HF(\Yn,m)\rightarrow...$$ where  $ \HF(\Yn,m)$  denotes the direct sum of the Floer homology groups associated to $\SpinC$ structures on $\Yn$ which extend over  the negative definite $2$-handle cobordism $W$ from $\Yn$ to $Y$ to   $\spinct\in \SpinC(W)$ satisfying $\langle c_1(\spinct),[\widehat{F}]\rangle + n = 2m.$   Note we have stated the splitting in a somewhat more concrete form than \cite[Theorem 9.19]{propOS}, implicitly using \cite[Section 7; particularly Lemma 7.10]{absgradOS}.   We also note that the exact sequence further decomposes along $\spinc\in \SpinC(Y)$, but we will not need this structure.

We use a generalization of this exact sequence to the case of knot Floer homology.   Let $K\subset Y$ be a null-homologous knot, and let $\mu\subset Y$ denote its meridian.  We can view $\mu$ as knot in each of the three $3$-manifolds of the sequence above, and consider their knot Floer homologies.  Note that $\mu\subset Y$ is an unknot, and $\mu\subset \Yzero$ (resp $\mu\subset  \Yn$) is isotopic to the core of the surgery, $\corezero$ (resp. $\coren$).  We have an exact sequence relating the knot Floer homology groups of these three knots
$$...\rightarrow \bigoplus_{i=1}^{n}\HFa(Y)\rightarrow \HFK(\Yzero,\corezero)\rightarrow \HFK(\Yn,\coren)\rightarrow...$$
where the first term is simply the Floer homology of $Y$, as $\mu$ is unknotted in this manifold.  While such an exact sequence has not, to our knowledge, appeared explicitly in the literature, it is  implicit from \ons's proof and nearly explicit in \cite{Eftekhary2}.  In any event, the sequence is easily obtained by adding an  additional basepoint in the handle region of the Heegaard quadruple diagram where the surgery curve is being varied (recall Figure \ref{fig:winding}). It is then straightforward to go through the now standard technique for proving the existence of surgery exact sequences (see, for instance \cite[Proof of Theorem 4.5]{branchedOS}), requiring that all differentials, chain maps, chain homotopies, etc. are defined by counting J-holomorphic Whitney polygons which avoid both basepoints.  As with the case of the Floer homology of closed $3$-manifolds, we have a splitting of this exact sequence into $n$ sequences according to the $\Spinc$ structures on $Y_n$:
$$...\rightarrow\HFa(Y)\longrightarrow \!\!\!\! \underset{\{\spinc |  \langle c_1(\spinc),[\widehat{F}]\rangle =2m \ \mathrm{mod}\ 2n \}}\bigoplus \!\!\!\! \HFK(\Yzero,\corezero,\spinc)\longrightarrow \HFK(\Yn,\coren,m)\rightarrow...$$
In addition, we know that the maps in the exact sequence are defined by counting $J$-holomorphic Whitney triangles associated to a doubly-pointed Heegaard triple diagram.  In each case there is a  $4$-manifold naturally associated to the triple diagram, and the first map is a sum over the triangle maps associated to homotopy classes whose $\SpinC$ structure extends over the cobordism to  $\spinc\in \SpinC(\Yzero)$ satisfying $\langle c_1(\spinc),[\widehat{F}]\rangle =2m$.  In particular, the component of the  map coming from a fixed homotopy class of triangles is independent of $n$.  Note that while these chain maps are likely an invariant of the embedded cylinder  in the cobordism coming from the trace of  $\mu$, we are not using this. We only use that  the Heegaard triple diagram defining the first map is independent of $n$.  

Given these exact sequences, we now apply Theorem \ref{thm:core2}.  This tells us that \begin{equation}\label{eq:HFKlarge} \HFK(\Yn,\coren,m)\cong H_*(S_m)\oplus H_*(Q_m),\end{equation} for sufficiently large $n$.  The exact sequence, however, tells us that this group is also the homology of the mapping cone of $$ \sum_{\spinct_m} \widehat{F}_{W_{\spinct_m}} :\HFa(Y) \longrightarrow  \!\!\!\! \underset{\{\spinc |  \langle c_1(\spinc),[\widehat{F}]\rangle =2m \ \mathrm{mod}\ 2n\}}\bigoplus \!\!\!\! \HFK(\Yzero,\corezero,\spinc)$$
where the sum is over all $\SpinC$ structures on the $2$-handle cobordism whose Chern class is congruent to $2m$, modulo  $2n$, and $\widehat{F}_{W_{\spinct_m}}$ is the map defined by counting $J$-holomorphic triangles representing these $\SpinC$ structures whose domains avoid both basepoints.  

The groups $\HFK(\Yzero,\corezero)$ were first studied by Eftekhary \cite{Eftekhary}, who referred to them as the {\em longitude Floer homology} groups.  He showed  \cite[Theorem 1.1]{Eftekhary} that they satisfy an adjunction inequality,  stating that $\HFK(\Yzero,\corezero,\spinc)=0$, unless \begin{equation}\label{eq:adj} -2g+2\le \langle c_1(\spinc),[\widehat{F}]\rangle \le 2g.\end{equation} 
Here $g$ denotes the minimal genus of any Seifert surface in the relative homology class of a fixed surface $F$, and $\widehat{F}$ denotes this latter surface capped off by the disk in the solid torus of the zero surgery.  Note that we have only stated the adjunction {\em inequality} aspect of \cite[Theorem 1.1]{Eftekhary} which in fact says that the bounds above are sharp.  Note, too, that  our inequality is asymmetric, due to the fact that we used  the map $\spinc_w(-): \Ta\cap \Tb\rightarrow \SpinC(\Yzero)$ coming from the basepoint $w$, whereas \cite{Eftekhary} uses the average $\OneHalf(c_1(\spinc_w(-))+c_1(\spinc_z(-)))$,  obtaining a symmetric inequality.  The important aspect of the inequality is that it implies there are at most $2g$ distinct $\SpinC$ structures on $\Yzero$ for which the middle term in the exact sequence is non-trivial.  It follows that for $n\ge 2g$, the groups under consideration $\HFK(\Yn,\coren,m)$, are isomorphic to the mapping cone of $$  \widehat{F}_{W_{\spinct_m}} :\HFa(Y) \longrightarrow  \HFK(\Yzero,\corezero,\spinc_m),$$
where $\spinct_m$, $\spinc_m$ are the $\SpinC$ structures on the cobordism and zero surgery, respectively, whose Chern classes satisfy Equation \eqref{eq:adj}.  Since these maps are independent of $n$, it follows that Equation \eqref{eq:HFKlarge} holds for all $n\ge 2g$.  Note, however, that the groups above are  the knot Floer homology groups associated to all relative $\SpinC$ structures on $\Yn\setminus \coren$ which project to $\spinc_m\in \SpinC(Y_n)$, under \eqref{eq:filling}.  Since our description of the contact invariant is in terms of the differential on the spectral sequence which starts at these groups and converges to $\HFa(\Yn,m)$, we must show that the filtration of $\CFa(\Yn,m)$ induced by $\coren$ agrees with the description of Theorem \ref{thm:core2}. (Note that Equation \eqref{eq:HFKlarge} states only that 
the associated graded homology groups agree.  We need to understand the entire filtration, and not simply the $E^1$-term of the corresponding spectral sequence.)
 In the case at hand, however, identification of filtrations  is immediate.  We are interested in  the inclusion of the bottom subcomplex of the knot Floer homology filtration  into the  Floer homology of  $\Yn$ when $\coren$ is rationally fibered; 
 namely, we would like to know whether the image of the inclusion map
 $$\iota: \F(Y, K,\mathrm{bottom}) \hookrightarrow \CF(\Yn,m)$$ is a boundary.
 
 Since the bottom subcomplex has rank one homology, this is determined by the homologies $\HFa(\Yn,m)$ and $$\HFK(\Yn,\coren,m)\cong H_*(S_m)\oplus H_*(Q_m)\cong \FF\oplus H_*(Q_m).$$ Indeed, the  adjunction  argument given above shows that  for $n\geq 2g$, $\HFK(\Yn,\coren,m)$ can have at most two Alexander gradings with non-trivial knot Floer homology, i.e. the filtration has at most two steps, 
 with the bottom subcomplex of rank 1;  then, the differential $\d_1$ on $\HFK$ computes $\HFa(\Yn,m)$ and is determined 
 by the image of $\iota$ in homology (and conversely, $\d_1$ determines this image).  The group $\HFa(\Yn,m)$, however, is independent of $n$ once $n\ge 2g-1$, by \cite[Remark 4.3]{knotOS}, 
 so $\d_1$ should be the same for all $n\geq 2g$. This completes the proof.

\end{proof}

\begin{remark} The key ingredient in our proofs of Theorems  \ref{thm:nonvanishing} and \ref{2g} is the understanding of the filtered chain homotopy 
type of $\CFK(\Yn(K), K_n)$. For surgeries on a knot in $S^3$ (or more generally, in an integer homology L-space), this filtered chain complex can be understood via bordered Floer homology \cite[Sections 10, 11]{LOT}; in fact, the techniques 
of \cite{LOT} provide the answer for an arbitrary knot and arbitrary surgery coefficient. However, \cite{LOT} doesn't provide the answer for knots in an arbitrary 3-manifold $Y$; in any case, we find that a simple direct argument works better for our purposes.    

\end{remark}
  
\subsection{From integer to rational surgeries}

We have established Theorem \ref{p/q-surgery} for the case of integral surgery. The following lemma extends Theorem \ref{p/q-surgery} to rational surgeries. The proof of this lemma was explained to us by John Etnyre and Jeremy Van Horn-Morris.

\begin{lemma} \label{n-to-p/q} Let $(Y,K)$ be an open book decomposition compatible with the contact structure $\xi$. 
If  $p/q>n>0$, the contact manifold $(Y_{p/q}, \xi_{p/q})$ can be obtained from $(Y_n, \xi_n)$ by Legendrian surgery on a link.  
\end{lemma}

Since the contact invariant is natural with respect to Legendrian surgeries, we have 

\begin{cor} If $p/q>n>0$ and  $c(\xi_n)$ does not vanish, then $c(\xi_{p/q})$  is non-zero.
\end{cor}

\begin{proof}[Proof of Lemma \ref{n-to-p/q}] We will prove the lemma by doing Legendrian surgery in certain thickened tori $T^2 \times [0,1]$ inside 
$(Y_n, \xi_n)$.
(More precisely, $T^2$ here stands for the boundary of a tubular neighborhood of the binding.) Tight contact structures on $T^2 \times [0,1]$ were classified in \cite{Ho}.

To begin, consider an honest open book $(Y, K)$ with the induced contact structure $\xi$. Remove a small neighborhood of $K$ 
with convex boundary. For the torus $T^2= \d (Y\setminus \nbdK)$ (oriented as the boundary of $Y\setminus\nbdK$), 
fix the identification $T^2=\RR^2/\ZZ^2$ so that the longitude corresponds to $(1, 0)$, and the meridian to $(0,1)$.
There are two parallel dividing curves on this torus; let $s_0$ denote their slope. (We write $s=y/x$ for the slope corresponding to $(x, y)$.)
Notice that since $K$ is a transverse knot, we can assume that $s_0= -n_0$ for some integer $n_0>0$. 
(The number $n_0$ gets larger if we choose a smaller 
neighborhood of $K$.)    

We will perform $n$-surgery on $(Y, \xi)$ by adding ``extra rotation'' in the neighborhood of the binding. Consider a block 
$T^2 \times [0,1]$ with a tight, minimally twisting, positively co-oriented contact structure $\zeta$ whose dividing curves 
rotate linearly from slope $s_0$ on $ T_0=T^2\times \{0\}$ through larger negative slopes, vertical slope, and then through large positive slopes to  
$s_1=n$. (This $T^2 \times [0,1]$ with ``linearly rotating'' contact structure
is isomorphic to a subset of the standard Stein fillable 3-torus.) 
    Attach this $T^2 \times [0,1]$ block to  $Y\setminus \nbdK$ so that $T^2=\d (Y\setminus \nbdK)$ is glued to $T_0$, 
and the dividing curves match. Now, $T_1=T^2\times \{1\}$ becomes the boundary torus; we can perturb this convex torus so that it becomes pre-Lagrangian,
with the linear characteristic foliation given by curves of slope $n$. The fibration of $Y\setminus \nbdK$ 
by the pages of the open book $(Y, K)$ extends into $T^2 \times [0,1]$  (compatibly with the contact structure). 
Collapsing to a point each leaf of the foliation of $T_1$, we get the surgered manifold $Y_n$, 
equipped with a well-defined contact structure and an open book decomposition. The contact structure is isotopic to $\xi_n$ and compatible 
with the open book: this is clear away from the binding, and we know that a contact structure extends uniquely over the binding  \cite{BEV}.

We will now perform Legendrian surgeries inside the block  $T^2 \times [0,1] \subset (Y\setminus \nbdK) \cup_{T_0} T^2 \times [0,1]$ 
to change the slopes on $T_1$ to $p/q$. After collapsing $T_1$ to a circle in the resulting contact manifold, we will get an open book compatible with $(Y_{p/q}, \xi_{p/q})$, together with a sequence of Legendrian surgeries that produce $(Y_{p/q}, \xi_{p/q})$ out of  $(Y_n$, $\xi_n)$.

Contact structures on $T^2 \times [0,1]$ were studied by Honda in \cite{Ho}, and can be conveniently 
described using the Farey tessellation of the unit disk \cite[Section 3.4.3]{Ho}. 
By Honda's work, the contact structures we are interested in decompose into ``bypass layers'' as dictated by the Farey tessellation 
and the boundary slopes. For the linearly rotating contact structures we considered above, all the bypass layers have negative sign. (We will not explain these terms here; the reader is referred to \cite{Ho} for all the details.)
The tessellation picture (Figure \ref{farey}) will also help 
to keep track of the transformation of the boundary slope of $T_1$  under Legendrian surgeries. 
Our toric block $T^2 \times [0,1]$ corresponds to the arc of the unit circle sweeping clockwise from $-\frac{n_0}1$ to $\frac{n}1$;
thus we will be focusing on the left side of the tessellation disk.

 \begin{figure}[htb] 
 \bigskip
  \bigskip
	\labellist
	\small\hair 2pt
	\pinlabel $\frac01$ at 170 81
	\pinlabel $\frac10$ at -5 81
	\pinlabel $\frac11$ at 85 170
	\pinlabel $-\frac11$ at 80 -9
	\pinlabel $\frac12$ at 149 143
	\pinlabel $-\frac12$ at 144 17
	\pinlabel $\frac21$ at 23 142
	\pinlabel $-\frac{2}1$ at 15 17
	\pinlabel $\frac32$ at 53 162
	\pinlabel $\frac31$ at 4 113
	\pinlabel $\frac53$ at 40 156
		\pinlabel $\frac{12}7$ at 31 149
	\endlabellist
\centering	
\includegraphics[scale=1.0]{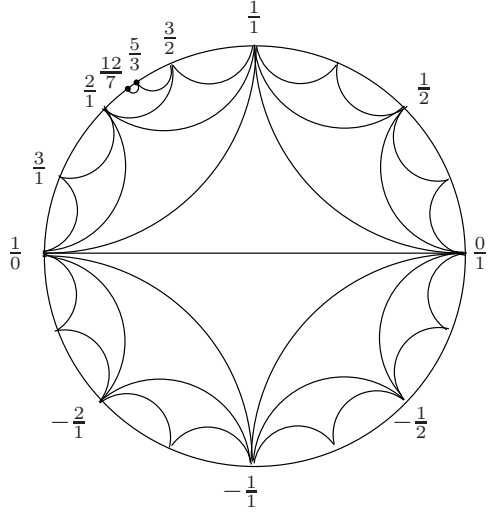} 
\bigskip
\caption{The Farey tessellation.}
\label{farey} 
\end{figure}
  
The results of \cite{Ho} imply  that  if the boundary slopes of a toric slice are $s_0$ and $s_1$, then for any given rational slope $s$ between $s_0$ and $s_1$ there  exists a pre-Lagrangian torus $T_s$ in  $T^2 \times [0,1]$ such that every leaf of its (linear) characteristic foliation has slope $s$. (In our case, $s_1$ is always greater than $s_0$, so $s$ can vary in $(-\infty, s_0]  \cup (s_1, \infty)$; this means that $s$ lies on the clockwise 
arc from $s_0$ to $s_1$.)

 We will be performing Legendrian surgeries on leaves of such foliations 
(with appropriate slopes). The key observation is that, if inside a block with the back slope $s_1=\frac{w}t$ we 
carry out a Legendrian surgery on a leaf with slope $s=\frac{u}v>s_1$, such that there is a tessellation edge from $s$ to $s_1$, then after surgery the back slope will be $s'_1=\frac{u+w}{v+t}$. (In other words, the new slope $s_1$ is the midpoint of the arc between $s$ and $s_1$, and can be reached 
from $s_1$ by hopping to the left along a shorter edge.) The slope transforms this way because the surgery can be interpreted as splitting along  
$T_s$ and regluing after a Dehn twist; the existence of an edge from $s$ to $s_1$ ensures that the curves corresponding to $(v, u)$ 
and $(t, w)$ intersect in $T^2$ homologically once, and thus after the Dehn twist, $s'_1$ must be the slope of the curve $(v+t, u+w)$.   We have found the boundary slopes of the resulting contact structure on $T^2\times [0,1]$, and thus can assert (using classification results of \cite{Ho}) that 
the contact structure decomposes into the bypass layers dictated by the Farey tessellation. We can conclude that it is isotopic to 
the corresponding linearly rotating contact structure if all the bypass layers are negative. But it is clear that at least some bypass layers remain negative, and if some other were positive, the resulting contact structure would be overtwisted. This cannot happen since our $T^2\times [0,1]$ embeds in a fillable contact manifold, both before and after Legendrian surgery.

Therefore, we have shown that Legendrian surgeries on leaves of the characteristic foliation on tori relate our model contact
structures to one another, changing the boundary slopes as predicted by the edges of the Farey tessellation.  
Now it remains to find the shortest sequence of edges connecting $\frac{n}1$ to $\frac{p}q$ in the tessellation picture, 
and perform the corresponding Legendrian surgeries. Suppose that $m$ is an integer such that $m+1>p/q>m$. If  $m > n$, the sequence starts with hopping from $\frac{n}1$ to $\frac{m}1$ through integer slopes. Each of these hops corresponds to  Legendrian surgery on a leaf in the pre-Lagrangian torus 
with slope $\frac{1}0$. Next, we continue along the edges from $m$ to $p/q$.

We illustrate this process by an example, describing  a sequence of Legendrian surgeries that produces $(Y_{12/7}, \xi_{12/7})$ from $(Y_{1/1}, \xi_{1/1})$. Constructing the point $\frac{12}7$ 
in the tessellation disk, we get from $\frac{1}1$ to $\frac{12}7$ by moving along three edges: the edge from $\frac{1}1$ to $\frac32$ (the midpoint of $\frac11$ and $\frac12$), then the edge from $\frac32$ to $\frac53$ (the midpoint of $\frac32$ and $\frac21$), then
the edge from  $\frac53$ to $\frac{12}7$ (the midpoint of $\frac53$ and $\frac{7}4$). These edges are shown on Figure \ref{farey}.
Therefore, $(Y_{12/7}, \xi_{12/7})$ can be obtained from $(Y_{1/1}, \xi_{1/1})$ by performing Legendrian surgery on the 3-component link 
consisting of two leaves of the characteristic foliation in the pre-Lagrangian torus with slope  $\frac21$, and a leaf of the foliation in the torus 
with slope $\frac{7}4$. The general case is treated similarly.

\end{proof}

\begin{remark} It is easy to see that under the hypotheses of Lemma \ref{n-to-p/q}, the manifold $Y_{p/q}$ carries a tight contact structure 
(with a non-vanishing invariant) for 
every $p/q>n$. Indeed, by the slam-dunk move \cite{GS}, performing  $p/q$-surgery on $K$ is equivalent to performing $n$-surgery on $K$, 
followed by  $r$-surgery on the meridian of $K$, where $r=\frac{q}{qn-p}$. Since $r<0$, by \cite{DGS} an $r$-surgery can be realized by a sequence 
of Legendrian surgeries, which results in a contact structure with non-vanishing contact invariant.

Lemma \ref{n-to-p/q} establishes a stronger result: a specific contact structure $\xi_{p/q}$, arising from the given open book,
 has non-vanishing contact invariant $c(\xi_{p/q})$. 
\end{remark}

\end{document}